\documentclass[11pt,reqno]{amsart}
\usepackage[T1]{fontenc}
\usepackage{lmodern}
\usepackage[a4paper,margin=29mm]{geometry}
\usepackage{amsmath,amssymb,amsthm,mathtools}
\usepackage{microtype}
\usepackage{enumitem}
\usepackage{xcolor}
\usepackage[colorlinks=false,allcolors=black,linktoc=all,pdfborder={0 0 1}]{hyperref}
\hypersetup{pdftitle={Brake orbit multiplicity on real prequantization bundles},pdfauthor={Peng Liu and Dong Xie},pdfsubject={Brake orbit bounds and curvature pinching on Finsler spheres}}
\newcommand{\doilink}[2]{%
  \begingroup
  \hypersetup{pdfborder={0 0 0}}%
  \href{#1}{#2}%
  \endgroup}
\numberwithin{equation}{section}
\newtheorem{theorem}{Theorem}[section]
\newtheorem{proposition}[theorem]{Proposition}
\newtheorem{lemma}[theorem]{Lemma}
\newtheorem{corollary}[theorem]{Corollary}
\newtheorem{conjecture}{Conjecture}

\theoremstyle{definition}

\newtheorem{remark}[theorem]{Remark}
\newtheorem{example}[theorem]{Example}
\newcommand{\F}{{\mathbb F_2}}
\newcommand{\Z}{\mathbb Z}
\newcommand{\R}{\mathbb R}
\newcommand{\RFH}{\operatorname{RFH}}
\newcommand{\QH}{\operatorname{QH}}
\newcommand{\RS}{\mathrm{RS}}
\newcommand{\Id}{\mathrm{Id}}
\newcommand{\Crit}{\operatorname{Crit}}
\newcommand{\ind}{\operatorname{ind}}
\newcommand{\rank}{\operatorname{rank}}
\newcommand{\Span}{\operatorname{span}}
\DeclareMathOperator{\Fix}{Fix}
\newcommand{\cC}{\mathcal C}
\setlist[enumerate]{leftmargin=*,itemsep=4pt}
\title[Brake orbits on real prequantization bundles]{Multiplicity of brake Reeb orbits on real prequantization bundles}
\author[Peng Liu]{Peng Liu}
\address[Peng Liu]{School of Mathematical Sciences, Nankai University,
Tianjin 300071, China}
\email{\href{mailto:liupeng@nankai.edu.cn}{liupeng@nankai.edu.cn}}
\author[Dong Xie]{Dong Xie}
\address[Dong Xie]{School of Mathematical Sciences, Nankai University,
Tianjin 300071, China; Centre de recherches math\'ematiques (CRM),
Universit\'e de Montr\'eal, Montr\'eal, QC H3C 3J7, Canada}
\email{\href{mailto:dong.xie@nankai.edu.cn}{dong.xie@nankai.edu.cn}}
\thanks{Corresponding author: Dong Xie. The authors contributed equally to this work.}
\date{}
\subjclass[2020]{Primary 37J46; Secondary 53D10, 53D40, 58E10}
\keywords{Brake orbits, real prequantization bundles,
Lagrangian Rabinowitz Floer homology, Maslov index, Finsler geodesics}

\begin{document}
\begin{abstract}
We prove multiplicity results for geometrically distinct brake Reeb
orbits on real prequantization $S^1$-bundles under geometric and
dynamical assumptions.
The lower bounds are expressed in terms of the sum of the mod-two
Betti numbers of the real locus of the base.
Our proofs combine Lagrangian Rabinowitz Floer homology with
the index iteration theory for brake orbits.
As an application, we obtain a sharp lower bound for the number
of brake geodesics on Finsler spheres with pinched flag curvature.
\end{abstract}
\maketitle

\setcounter{tocdepth}{2}
\begingroup
\hypersetup{pdfborder={0 0 0}}
\tableofcontents
\endgroup
\clearpage

\section{Introduction}

\subsection{Background and conjectures}
\label{subsec:conjectures}

We study the number of geometrically distinct brake Reeb orbits
on real prequantization bundles. We conjecture that this number
is at least the sum of the mod-two Betti numbers of the real
locus of the base.
We prove multiplicity results under additional geometric and
dynamical assumptions, with applications to brake geodesics
on Finsler spheres.

Let $Y$ be a contact manifold with contact form $\alpha$.
The \emph{Reeb vector field} $R_\alpha$ is defined by
\[
 \alpha(R_\alpha)=1,\qquad d\alpha(R_\alpha,\cdot)=0.
\]
We denote its flow by $\phi_\alpha^t$.
A \emph{Reeb orbit} is an integral curve of $R_\alpha$.
A \emph{closed Reeb orbit} of period $T>0$ is a smooth map
$\gamma:\R/T\Z\to Y$ satisfying
$\dot\gamma(t)=R_\alpha(\gamma(t))$.

Let $\rho:Y\to Y$ be a smooth involution
{($\rho^2=\Id_Y$)} satisfying $\rho^*\alpha=-\alpha$.
The triple $(Y,\alpha,\rho)$ is called a \emph{real contact manifold}.
A closed Reeb orbit is called a \emph{$\rho$-brake Reeb orbit}
if it admits a parametrization $\gamma:\R/T\Z\to Y$ satisfying
$\gamma(-t)=\rho(\gamma(t))$ for all $t\in\R/T\Z$.
Such orbits are called symmetric periodic Reeb orbits
in \cite{KKK}.
The orbit $\gamma$ is \emph{prime} if $T$ is its least positive
period.
Two closed Reeb orbits are \emph{geometrically distinct}
if their images in $Y$ are different.

A \emph{Reeb chord} of a Legendrian submanifold $K\subset Y$
is a path $x:[0,T]\to Y$, with $T>0$, satisfying
$\dot x(t)=R_\alpha(x(t))$ and $x(0),x(T)\in K$.
In a real contact manifold $(Y,\alpha,\rho)$, the fixed point
set $\Fix(\rho)$, when nonempty, is Legendrian.
Every Reeb chord $x:[0,T]\to Y$ of $\Fix(\rho)$ extends
by reflection to a $\rho$-brake Reeb orbit of period $2T$.
Conversely, a $\rho$-brake Reeb orbit $\gamma$ of period $T$,
{parametrized so that $\gamma(-t)=\rho(\gamma(t))$}, restricts to a Reeb chord on $[0,T/2]$.
Different Reeb chords may extend to brake orbits with the
same image.

Let $(B^{2n},\omega)$ be a closed connected symplectic
manifold, with $n\ge1$, and fix an integral lift
$[\omega]_{\mathbb Z}\in H^2(B;\mathbb Z)$ of $[\omega]$.
A \emph{prequantization $S^1$-bundle}
$\pi:(Y^{2n+1},\xi)\to(B^{2n},\omega)$ is a principal
$S^1$-bundle with Euler class $-[\omega]_{\mathbb Z}$,
equipped with a cooriented contact structure $\xi$
that admits a connection contact form $\alpha_0$
satisfying $\ker\alpha_0=\xi$ and
$d\alpha_0=\pi^*\omega$; see \cite{BKK,GGM}.
The form $\alpha_0$ induces the given coorientation
and is normalized so that the Reeb orbits along the
fibers have minimal period $1$.
For a fixed $S^1$-action and contact structure $\xi$,
this normalized connection form is unique.

A prequantization $S^1$-bundle
$\pi:(Y,\xi)\to(B,\omega)$ is called a
\emph{real prequantization $S^1$-bundle} if it is equipped
with a smooth involution $\rho:Y\to Y$ satisfying
$\rho^*\alpha_0=-\alpha_0$, where $\alpha_0$ is its
normalized connection contact form.
The involution $\rho$ reverses the Reeb flow of $\alpha_0$
and therefore maps fibers to fibers.
It induces an anti-symplectic involution $\tau:B\to B$
satisfying $\pi\circ\rho=\tau\circ\pi$.
The fixed point set $L=\Fix(\tau)$ is called the
\emph{real locus}; when nonempty, it is a closed
Lagrangian submanifold of $B$.

For results on the multiplicity of closed Reeb orbits on
prequantization bundles, see \cite{GGM,AM26a,AM26b}.
In the symplectically aspherical setting, Albers and
Frauenfelder~\cite{AFReebChords} obtained Reeb chord multiplicity
results for Bohr--Sommerfeld Legendrians in certain prequantization
bundles and for a class of $S^1$-invariant contact forms.
For prequantization bundles over closed monotone symplectic
manifolds with minimal Chern number at least $2$,
Shelukhin~\cite{ShelukhinChord} proved that every closed Legendrian
submanifold admits a Reeb chord for every supporting contact form.

All homology groups are taken with coefficients in $\F$.
For a closed manifold $M$, we write
$b_j(M)=\dim_\F H_j(M;\F)$ for its $j$th Betti number and
$b(M)=\sum_{j=0}^{\dim M}b_j(M)$ for its total Betti number.

For the connection contact form $\alpha_0$, the prime
$\rho$-brake Reeb orbits are precisely the circle fibers
over $L$.
We ask whether the topology of $L$ gives a lower bound
for the number of geometrically distinct $\rho$-brake
Reeb orbits of any supporting contact form satisfying
$\rho^*\alpha=-\alpha$.

In Lagrangian intersection theory, the Arnold--Givental
conjecture, recently proved in full generality by Bai,
Shelukhin, Wang, and Xu~\cite{BSWXAG}, asserts that a real
Lagrangian $L$ in a closed symplectic manifold has at least
$b(L)$ intersection points with any Hamiltonian image
intersecting it transversely.
By analogy, we propose the following conjecture for
brake Reeb orbits.

\begin{conjecture}\label{conj:real-prequantization-bound}
Let $\pi:(Y^{2n+1},\xi)\to(B^{2n},\omega)$ be a real
prequantization $S^1$-bundle over a closed connected
symplectic manifold, with $n\ge1$.
Let $\rho$ be its involution and $\tau$ the induced
anti-symplectic involution of $B$.
Assume that $\Fix(\rho)\ne\varnothing$, and put
$L=\Fix(\tau)$.
Let $\alpha$ be a contact form supporting $\xi$ and
satisfying $\rho^*\alpha=-\alpha$.
Then the Reeb flow of $\alpha$ has at least $b(L)$
geometrically distinct $\rho$-brake Reeb orbits.
\end{conjecture}

The contact forms considered in
Conjecture~\ref{conj:real-prequantization-bound} are precisely
the forms $\alpha=f\alpha_0$, where $f:Y\to(0,\infty)$ is
smooth and satisfies $f\circ\rho=f$.

For the Hopf fibration $S^{2n+1}\to\mathbb{CP}^n$,
complex conjugation $\rho(z)=\bar z$ descends to an
anti-symplectic involution of $\mathbb{CP}^n$ with real
locus $\mathbb{RP}^n$.
Let $\Sigma\subset\mathbb C^{n+1}$ be a smooth closed
$\rho$-invariant hypersurface that is \emph{strictly star-shaped}
with respect to the origin, meaning that every ray from
the origin meets $\Sigma$ exactly once and transversely.
In complex coordinates $z_j=x_j+iy_j$, the standard Liouville form
\[
 \lambda_{\mathrm{st}}
 =\frac12\sum_{j=0}^n(x_j\,dy_j-y_j\,dx_j)
\]
restricts to a contact form on $\Sigma$.
Radial projection is a contactomorphism to
the standard contact sphere that commutes with $\rho$.
Since $b(\mathbb{RP}^n)=n+1$,
Conjecture~\ref{conj:real-prequantization-bound} would imply
that $\Sigma$ carries at least $n+1$ geometrically distinct
$\rho$-brake Reeb orbits.
This is the contact version of Seifert's conjecture
for strictly star-shaped hypersurfaces;
see \cite{KKK}.

Giamb\`o, Giannoni, and Piccione~\cite{GGPSeifert} proved
Seifert's conjecture for natural Hamiltonian systems when
the potential well is homeomorphic to a ball and the energy
is a regular value of the potential.
Their result does not cover arbitrary strictly star-shaped
hypersurfaces.
For strictly star-shaped hypersurfaces invariant under complex
conjugation, Rabinowitz~\cite{RabinowitzBrake} proved the existence of a
brake orbit, and Szulkin~\cite{SzulkinBrake} obtained the lower
bound $n+1$ under a geometric pinching condition.
Under the additional assumptions of central symmetry and strict
convexity, Long, Zhang, and Zhu~\cite{LZZBrake} obtained at least two
brake orbits, and C.~Liu and Zhang~\cite{LZSeifert} subsequently
proved the lower bound $n+1$ without assuming nondegeneracy.
Z.~Liu and Zhang~\cite{LZDynamic} obtained the
same bound for centrally symmetric strictly star-shaped
hypersurfaces invariant under complex conjugation, assuming
dynamical convexity and nondegeneracy of all closed characteristics.
For a survey of brake orbit multiplicity results,
see \cite{LLZSurvey}.

A further special case of
Conjecture~\ref{conj:real-prequantization-bound} concerns
Finsler spheres.
We use the definition of a Finsler metric in
\cite[Section~1.1]{BCS}.
For a Finsler manifold $(M,F)$, let $F^*$ denote the dual
Finsler norm. We write
\[
 \begin{aligned}
 S_FM&=\{(q,v)\in TM:F(q,v)=1\},\\
 S_F^*M&=\{(q,p)\in T^*M:F^*(q,p)=1\}
 \end{aligned}
\]
for its unit tangent and unit cosphere bundles, respectively.

We denote the flag curvature of $F$ by $K_F$.
The metric $F$ is called \emph{reversible} if
$F(q,-v)=F(q,v)$ for all $(q,v)\in TM$,
and \emph{non-reversible} otherwise.
When $M$ is compact, we define the \emph{reversibility}
of $F$ by
\[
 \lambda=\max_{(q,v)\in S_FM}F(q,-v).
\]
Then $\lambda=1$ if and only if $F$ is reversible.

Let $\sigma:S^{n+1}\to S^{n+1}$ be a smooth involution
with $\Fix(\sigma)\ne\varnothing$.
Following \cite[Section~2.2]{KimSphere}, we call a Finsler metric
$F$ on $S^{n+1}$ a \emph{real Finsler metric} with respect to
$\sigma$ if
\begin{equation}\label{eq:sphere-conjecture-symmetry}
 F\bigl(\sigma(q),-(D\sigma)_q(v)\bigr)=F(q,v)
 \qquad\bigl(q\in S^{n+1},\ v\in T_qS^{n+1}\bigr).
\end{equation}

A unit-speed closed $F$-geodesic
$c:\R/T\Z\to S^{n+1}$ is called a \emph{$\sigma$-brake geodesic}
if, after a time translation, $c(-t)=\sigma(c(t))$
for all $t\in\R/T\Z$.
Such geodesics are called symmetric closed geodesics in
\cite[Section~2.2]{KimSphere}.
The geodesic $c$ is \emph{prime} if $T$ is its least positive period.
Two prime unit-speed closed $F$-geodesics $c_1$ and $c_2$
are called \emph{distinct} if there is no $s\in\R$ such that
$c_1(t)=c_2(t+s)$ for all $t\in\R$.
For a reversible metric $F$, two prime closed $F$-geodesics
are called \emph{geometrically distinct} if their images
in $S^{n+1}$ are different.

\begin{conjecture}\label{conj:sphere-brake-bound}
Let $F$ be a Finsler metric on $S^{n+1}$, $n\ge1$, and let
$\sigma:S^{n+1}\to S^{n+1}$ be the restriction of a nontrivial
orthogonal linear involution of $\R^{n+2}$.
Assume that $\Fix(\sigma)\ne\varnothing$ and that $F$ is
real with respect to $\sigma$.
In this case, $\Fix(\sigma)\cong S^d$ for some $0\le d\le n$.
Set $a=\min\{d,n-d\}$.
Then $F$ has at least $2a+2$ distinct prime
$\sigma$-brake geodesics.
\end{conjecture}

To see that Conjecture~\ref{conj:sphere-brake-bound} follows from
Conjecture~\ref{conj:real-prequantization-bound}, consider
the unit cosphere bundle
\begin{equation}\label{eq:sphere-finsler-contact}
 Y_F=S_F^*S^{n+1},\qquad
 \alpha_F=\frac{1}{2\pi}\theta|_{Y_F},\qquad
 \rho(q,p)=\bigl(\sigma(q),-\bigl((D\sigma)_q^{-1}\bigr)^*(p)\bigr),
\end{equation}
where $\theta$ is the canonical one-form on $T^*S^{n+1}$.
By \eqref{eq:sphere-conjecture-symmetry}, $\rho$ preserves $Y_F$.
Since $\rho^*\theta=-\theta$, the triple
$(Y_F,\alpha_F,\rho)$ is a real contact manifold.

{The Legendre transform and a constant rescaling of time
identify prime unit-speed $\sigma$-brake geodesics modulo
time translations with geometrically distinct $\rho$-brake
Reeb orbits of $\alpha_F$ in $Y_F$;
see \cite[Section~2.2]{KimSphere}.}

Let $g_0$ denote the round metric of curvature one on $S^{n+1}$.
The unit cosphere bundle $S_{g_0}^*S^{n+1}$, with normalized
connection contact form $\theta|_{S_{g_0}^*S^{n+1}}/(2\pi)$, is a
prequantization $S^1$-bundle over the complex quadric $Q^n$.
Since $\sigma$ preserves $g_0$, this is a real prequantization
$S^1$-bundle with involution $\rho$.
The induced anti-symplectic involution of $Q^n$ has real locus
\begin{equation}\label{eq:sphere-conjecture-real-locus}
 L\cong(S^d\times S^{n-d})/\bigl((x,y)\sim(-x,-y)\bigr),
 \qquad b(L)=2a+2.
\end{equation}
{Under the radial identification
$Y_F\cong S_{g_0}^*S^{n+1}$, which commutes with $\rho$,
we have
\[
 \alpha_F=f\,\frac{\theta|_{S_{g_0}^*S^{n+1}}}{2\pi},
\]
where $f>0$ is smooth and satisfies $f\circ\rho=f$.}
Hence Conjecture~\ref{conj:real-prequantization-bound} implies
Conjecture~\ref{conj:sphere-brake-bound}.
Here $\Fix(\rho)$ is the unit conormal bundle
of $\Fix(\sigma)$.
These identifications and the calculation of $b(L)$ are
explained in Section~\ref{subsec:sphere-contact-model}.

On cosphere bundles, Bro\'ci\'c, Cant, and
Shelukhin~\cite{BCSConormal} proved the existence of Reeb chords
on unit conormal bundles of closed submanifolds of positive
codimension in closed manifolds, for every supporting contact form.
Applied to $\Fix(\rho)$, this result yields a
$\rho$-brake Reeb orbit by reflection.
Bro\'ci\'c and Cant~\cite{BCReebChords} proved that if a closed
manifold admits a submersion from a product of spheres,
then every closed Legendrian submanifold of its cosphere
bundle admits a Reeb chord for every supporting contact form.

Corona~\cite{CoronaDisk} proved that a smooth
compact reversible Finsler manifold with boundary,
homeomorphic to an $N$-dimensional disk, $N\ge2$, has at least
$N$ geometrically distinct orthogonal geodesic chords,
provided that no geodesic chord is orthogonal to the boundary
at one endpoint and tangent to it at the other.

In our setting, suppose that $F$ is reversible and
$\Fix(\sigma)=S^n$.
By \eqref{eq:sphere-conjecture-symmetry}, $\sigma$ is an isometry
of $F$, so its fixed-point set $S^n$ is totally geodesic.
By uniqueness of geodesics, a geodesic segment whose interior
lies in an open hemisphere cannot meet the equator tangentially.
We can therefore apply this result to a closed hemisphere,
which is an $(n+1)$-dimensional disk, to obtain at least
$n+1$ geometrically distinct orthogonal geodesic chords.
Reflecting these chords across the equator yields at least
$n+1$ geometrically distinct prime $\sigma$-brake geodesics
on $S^{n+1}$.

For a Riemannian two-sphere with an orientation-reversing
isometric involution whose fixed set is a circle,
Kang~\cite{KangReversible} proved that there
are infinitely many symmetric closed geodesics if the
geodesic flow admits an invariant Birkhoff annulus as a
global surface of section.
S.~Kim~\cite{KimSphere} considered real Finsler
two-spheres with respect to the involutions specified there.
He proved that there are either exactly two or infinitely
many symmetric closed geodesics if
$(\lambda/(\lambda+1))^2<K_F\le1$.
Theorem~\ref{cor:sphere-pinching} below treats spheres of
dimension at least three.

\subsection{Main results}
\label{subsec:main-results}

We work in the setting of
Conjecture~\ref{conj:real-prequantization-bound}, with
normalized connection contact form $\alpha_0$.
Put $\Lambda=\Fix(\rho)$ and write the contact form as
$\alpha=f\alpha_0$.

Let $\delta=\phi_{\alpha_0}^{1/2}$ be the time-$1/2$
map of the Reeb flow of $\alpha_0$.
The restriction $\pi|_\Lambda:\Lambda\to L$ is a two-fold
covering, with $\Lambda$ possibly disconnected.
The involution $\delta$ exchanges the two points in each fiber.
If $f\circ\delta=f$, then $\delta^*\alpha=\alpha$, so the
Reeb flow of $\alpha$ descends to the quotient
$Y/\langle\delta\rangle$.

A Lagrangian $L\subset B$ is \emph{positively monotone} if there is
$\lambda_L>0$ such that
\begin{equation}\label{eq:monotone}
 \omega(A)=\lambda_L\mu_L(A)\qquad(A\in\pi_2(B,L)),
\end{equation}
where $\mu_L:\pi_2(B,L)\to\Z$ is the Maslov homomorphism.
The \emph{minimal Maslov number} $N_L$ of $L$ is the positive generator
of the image of $\mu_L$, with $N_L=\infty$ when that image is zero.

For a positively monotone $L$ with $2\le N_L<\infty$,
we write $\QH_*(L;\F)$ for the Lagrangian quantum homology
with coefficients in $\F[t,t^{-1}]$, where $|t|=-N_L$.
Following \cite{BC}, we say that $L$ is \emph{wide} over $\F$
if there is an isomorphism of graded $\F[t,t^{-1}]$-modules
\begin{equation}\label{eq:wide-module}
 \QH_*(L;\F)\cong H_*(L;\F)\otimes_{\F}\F[t,t^{-1}].
\end{equation}
We use the homological grading of \cite{BKK}.

A Reeb chord of $\Lambda$ is \emph{relatively contractible}
if it is homotopic to a constant path through paths with
endpoints on $\Lambda$.
We write $\cC^0_+(\alpha,\Lambda)$ for the set of relatively
contractible Reeb chords of $\Lambda$.
A chord $x:[0,T]\to Y$ is
\emph{nondegenerate} if
\begin{equation}\label{eq:real-chord-nondegeneracy}
 (D\phi_\alpha^T)_{x(0)}\bigl(T_{x(0)}\Lambda\bigr)\cap T_{x(T)}\Lambda=\{0\}
 \quad\text{in }\xi_{x(T)}.
\end{equation}
If a Reeb chord of $\Lambda$ is relatively contractible,
then the closed orbit obtained by reflection is contractible; see
Section~\ref{subsec:geometric-conventions}.

For a closed Reeb orbit $\gamma:\R/T\Z\to Y$, the
\emph{linearized Poincar\'e return map} is the symplectic
linear map
\[
 \left.(D\phi_\alpha^T)\right|_{\xi_{\gamma(0)}}:
 \xi_{\gamma(0)}\longrightarrow\xi_{\gamma(0)}.
\]
The orbit $\gamma$ is called \emph{nondegenerate} if this
map has no eigenvalue $1$.
A prime closed Reeb orbit is called \emph{strongly nondegenerate}
if all its iterates are nondegenerate.

When $L$ is positively monotone, we associate two Maslov
indices $\mu_I(x)$ and $\mu_{-I}(x)$ to each
relatively contractible Reeb chord $x$ whose associated
$\rho$-brake Reeb orbit
is nondegenerate.
These indices are defined from the linearized Reeb flow,
using the $+1$ and $-1$ eigenspaces of $(D\rho)|_\xi$ at the
endpoints of $x$ as Lagrangian boundary conditions, respectively.
The $+1$ eigenspaces are the tangent spaces of $\Lambda$.
We follow the notation of \cite[Section~4.1.3]{KKK}.
{Section~\ref{subsec:geometric-conventions} specifies the
capping conventions, and Lemma~\ref{lem:real-two-boundary-indices}
proves that both indices are independent of the capping
and adapted trivialization.}

{The first theorem gives the lower bound $b(L)$ when $N_L>n$.
For $N_L=n$, Theorem~\ref{thm:main-brake-critical} gives $b(L)-1$
under weaker index bounds and $b(L)$ under an additional mean
index assumption.}

\begin{theorem}\label{thm:main-brake-total-betti}
Let $\pi:(Y^{2n+1},\xi)\to(B^{2n},\omega)$ be a real
prequantization $S^1$-bundle over a closed connected symplectic
manifold, with $n\ge2$.
Let $\alpha_0$ be its normalized connection contact form,
$\rho$ its involution, and $\tau$ the induced anti-symplectic
involution of $B$.
Put $L=\Fix(\tau)$, $\Lambda=\Fix(\rho)$, and
$\delta=\phi_{\alpha_0}^{1/2}$.

Assume that $\Lambda\ne\varnothing$ and that $L$ is connected,
positively monotone, and wide over $\F$, with finite minimal
Maslov number $N_L>n$.
Let $\alpha=f\alpha_0$, where $f:Y\to(0,\infty)$ is smooth and
satisfies $f\circ\rho=f$ and $f\circ\delta=f$.
Suppose that every prime $\rho$-brake Reeb orbit of $\alpha$
carrying a relatively contractible Reeb chord of
$\Lambda$ is strongly nondegenerate.
Assume that every $x\in\cC^0_+(\alpha,\Lambda)$ satisfies
\begin{equation}\label{eq:main-brake-total-index}
 \mu_I(x),\ \mu_{-I}(x)\ge\frac n2+1
 \quad\text{or}\quad
 \mu_I(x),\ \mu_{-I}(x)\le-\frac n2-1.
\end{equation}
Then the Reeb flow of $\alpha$ has at least $b(L)$
geometrically distinct $\rho$-brake Reeb orbits.
\end{theorem}

Let $x$ be a relatively contractible Reeb chord
of $\Lambda$, and let $x^m$ denote its $m$th \emph{iterate},
obtained by reflection and concatenation.
Suppose that all iterates of the associated $\rho$-brake Reeb
orbit $x^2$ are nondegenerate.
Following \cite{KKK}, we define the \emph{mean index} of $x$ by
\begin{equation}\label{eq:brake-mean-definition}
 \widehat\mu(x)=\lim_{m\to\infty}\frac{\mu_I(x^m)}{m}.
\end{equation}
This is half the mean Conley--Zehnder index of $x^2$.
{For the iteration conventions, see
Section~\ref{subsec:index-iteration-tools}.}

The case $n=2$ of the next theorem requires the following
filling condition.
\begin{enumerate}[label=\textup{(F)},ref=F]
\item\label{ass:fillings}
The quotient $\bar Y=Y/\langle\delta\rangle$ admits a Liouville
filling $W$ containing an exact Lagrangian filling $C$ of
$\bar\Lambda=\Lambda/\langle\delta\rangle$.
The inclusion-induced map
\[
 \pi_1(\bar Y,\bar\Lambda)\longrightarrow\pi_1(W,C)
\]
is injective, and $\mu_C|_{\pi_2(W,C)}=0$.
Here $\mu_C$ is the Maslov class of $C$, and the relative homotopy
sets are based at a point of $\bar\Lambda$.
\end{enumerate}

\begin{theorem}\label{thm:main-brake-critical}
Let $\pi:(Y^{2n+1},\xi)\to(B^{2n},\omega)$ be a real
prequantization $S^1$-bundle over a closed connected symplectic
manifold, with $n\ge2$.
Let $\alpha_0$ be its normalized connection contact form,
$\rho$ its involution, and $\tau$ the induced anti-symplectic
involution of $B$.
Put $L=\Fix(\tau)$, $\Lambda=\Fix(\rho)$, and
$\delta=\phi_{\alpha_0}^{1/2}$.

Assume that $\Lambda\ne\varnothing$ and that $L$ is connected,
positively monotone, and wide over $\F$, with minimal Maslov
number $N_L=n$.
Let $\alpha=f\alpha_0$, where $f:Y\to(0,\infty)$ is smooth and
satisfies $f\circ\rho=f$ and $f\circ\delta=f$.
If $n=2$, assume in addition that
condition~\eqref{ass:fillings} holds.
Suppose that every prime $\rho$-brake Reeb orbit of $\alpha$
carrying a relatively contractible Reeb chord of
$\Lambda$ is strongly nondegenerate.
Assume that every $x\in\cC^0_+(\alpha,\Lambda)$ satisfies
\begin{equation}\label{eq:main-brake-critical-index}
 \mu_I(x),\ \mu_{-I}(x)\ge\frac n2
 \quad\text{or}\quad
 \mu_I(x),\ \mu_{-I}(x)\le-\frac n2.
\end{equation}
Then the Reeb flow of $\alpha$ has at least $b(L)-1$
geometrically distinct $\rho$-brake Reeb orbits.
If, in addition,
\begin{equation}\label{eq:main-brake-critical-mean}
 |\widehat\mu(x)|>\frac{n\bigl(b(L)-1\bigr)}{b(L)}
 \qquad\bigl(x\in\cC^0_+(\alpha,\Lambda)\bigr),
\end{equation}
then the Reeb flow of $\alpha$ has at least $b(L)$
geometrically distinct $\rho$-brake Reeb orbits.
\end{theorem}

We first consider the Hopf fibration $S^{2n+1}\to\mathbb{CP}^n$,
with $n\ge2$. The real locus $L=\mathbb{RP}^n$ is connected, positively
monotone, and wide over $\F$, with $N_L=b(L)=n+1$;
see \cite[Proposition~1.6]{BKK}.
For the normalized Hopf connection contact form
$\alpha_{\mathrm H}=\lambda_{\mathrm{st}}|_{S^{2n+1}}/\pi$,
the map $\delta=\phi_{\alpha_{\mathrm H}}^{1/2}$ is the antipodal map.
For a domain $D$ as in the following corollary, write its radial
parametrization as $\iota_D(u)=r(u)u$, where $r(u)>0$. Then
\[
 \iota_D^*\alpha_D=\pi r^2\alpha_{\mathrm H},\qquad
 r(\bar u)=r(u)=r(-u).
\]
This parametrization identifies $S^n$ with $\Lambda_D$ and
conjugates the Reeb flow of $\pi r^2\alpha_{\mathrm H}$ to
that of $\alpha_D$, preserving nondegeneracy and the Maslov
indices computed using relative cappings.
Since $S^{2n+1}$ is simply connected and $S^n$ is connected,
every chord of $\Lambda_D$ is relatively contractible.
Theorem~\ref{thm:main-brake-total-betti} therefore gives the
following corollary.

\begin{corollary}\label{cor:star-shaped-brake}
Let $n\ge2$ and let $D\subset\mathbb C^{n+1}$ be a smooth bounded
domain containing the origin and strictly star-shaped with respect
to it. Assume that $D$ is invariant under complex conjugation
$\rho(z)=\bar z$ and that $\partial D$ is centrally symmetric.
Put $\alpha_D=\lambda_{\mathrm{st}}|_{\partial D}$ and
$\Lambda_D=\partial D\cap\mathbb R^{n+1}$.
Assume that every prime $\rho$-brake Reeb orbit of $\alpha_D$ is
strongly nondegenerate and that every Reeb chord
$x$ of $\Lambda_D$ satisfies
\begin{equation}\label{eq:star-shaped-brake-index}
 \mu_I(x),\ \mu_{-I}(x)\ge\frac{n+2}{2}.
\end{equation}
Then the Reeb flow of $\alpha_D$ has at least $n+1$
geometrically distinct $\rho$-brake Reeb orbits.
\end{corollary}

In this setting, condition~\eqref{eq:star-shaped-brake-index}
is the real dynamical convexity condition
of \cite[Definition~4.4]{KKK}.
The lower bound is attained by irrational ellipsoids, whose prime
closed Reeb orbits are strongly nondegenerate;
see \cite[p.~1199]{LZIteration}.
Since these ellipsoids are strictly convex and invariant
under complex conjugation, \cite[Theorem~4.5]{KKK} implies
that their Reeb chords satisfy
\eqref{eq:star-shaped-brake-index}.

\begin{remark}\label{rem:star-shaped-comparison}
J.~Kim, S.~Kim, and Kwon~\cite[Theorem~1.2]{KKK} state the lower
bound $n+1$ under real dynamical convexity and nondegeneracy
of all closed Reeb orbits, without assuming central symmetry.
The calculation in \cite[Section~3.8]{KKK} treats nondegenerate
Reeb chords as pairs of distinct chords under {the involution defined in \cite[(3.7)]{KKK}}.
In the anti-symplectic case, however, nondegeneracy does not
exclude fixed chords; in particular, even iterates of Reeb
chords are fixed by this involution.
{The degree relation in \cite[(3.21)]{KKK}} is not valid in the anti-symplectic case.
Consequently, under the hypotheses of their Theorem~1.2,
the stated lower bound does not follow from the argument
in the proof of \cite[Theorem~4.12]{KKK}.

Under central symmetry, the antipodal map gives the free
action $x\mapsto-x$ on Reeb chords.
For this action, Corollary~3.35 and Theorem~3.40
of~\cite{KKK} apply.
Suppose that there are only finitely many geometrically distinct
prime brake orbits, and denote their number by $s$.
Their $2s$ half-period chords form $s$ antipodal pairs.
Applying the index argument in the proof of
\cite[Theorem~4.12]{KKK} to representatives of these pairs
gives $s\ge n+1$ under \eqref{eq:star-shaped-brake-index}.
The argument requires nondegeneracy only for brake orbits
and their iterates.
\end{remark}

We next consider Finsler spheres.
A closed $F$-geodesic is \emph{nondegenerate} if its corresponding
closed Reeb orbit of $\alpha_F$ is nondegenerate.
A prime closed $F$-geodesic is \emph{strongly nondegenerate}
if all its iterates are nondegenerate.

\begin{theorem}\label{cor:sphere-pinching}
Let $F$ be a Finsler metric on $S^{n+1}$, $n\ge2$, with
reversibility $\lambda$, and let
$\sigma:S^{n+1}\to S^{n+1}$ be the restriction of a nontrivial
orthogonal linear involution of $\R^{n+2}$.
Assume that $\Fix(\sigma)\ne\varnothing$ and that $F$ is
real with respect to $\sigma$.
In this case, $\Fix(\sigma)\cong S^d$ for some $0\le d\le n$.
Set $a=\min\{d,n-d\}$.
If $a\ge1$, assume in addition that
\begin{equation}\label{eq:sphere-antipodal-symmetry}
 F(-q,-v)=F(q,v).
\end{equation}
Assume that every prime $\sigma$-brake geodesic of $F$ is
strongly nondegenerate.
If the flag curvature satisfies
\begin{equation}\label{eq:sphere-standard-pinching}
 \left(\frac{\lambda}{\lambda+1}\right)^2<K_F\le1,
\end{equation}
then $F$ has at least $2a+1$ distinct prime $\sigma$-brake geodesics.
If, in addition,
\begin{equation}\label{eq:sphere-pinching}
 \frac{(2a+1)^2}{(a+1)^2}
 \left(\frac{\lambda}{\lambda+1}\right)^2
 <K_F\le1,
\end{equation}
then $F$ has at least $2a+2$ distinct prime $\sigma$-brake geodesics.
\end{theorem}

The Katok--Ziller type metrics in
Example~\ref{ex:sphere-katok-sharp} show that the lower bound
$2a+2$ in Theorem~\ref{cor:sphere-pinching} and
Conjecture~\ref{conj:sphere-brake-bound} is sharp for every
possible dimension $0\le d\le n$ of the fixed point set.

\begin{remark}
\leavevmode When $a=0$, \eqref{eq:sphere-standard-pinching}
and \eqref{eq:sphere-pinching} coincide, so
Theorem~\ref{cor:sphere-pinching} gives at least two distinct
prime $\sigma$-brake geodesics without an antipodal
invariance assumption.
For $a\ge1$, \eqref{eq:sphere-pinching} implies
$\lambda<(a+1)/a$.
\end{remark}

\subsection{Outline of the proofs}
\label{subsec:proof-outline}

We first pass to the quotient by $\delta$, where the image
$\bar\Lambda$ of $\Lambda$ projects diffeomorphically onto $L$.
For $N_L\ge3$, the comparison of Bae, Kang, and
S.~Kim~\cite[Theorem~1.1]{BKK} identifies the Lagrangian Rabinowitz
Floer homology of $\bar\Lambda$ with the quantum homology of $L$.
For $N_L=2$, we prove the same isomorphism under
condition~\eqref{ass:fillings} in
Lemma~\ref{lem:relative-filling-comparison}.
When $L$ is wide, we can therefore compute the dimension
of RFH in each degree from the Betti numbers of $L$.

We may assume that there are only finitely many geometrically
distinct $\rho$-brake Reeb orbits.
Every relatively contractible Reeb chord of
$\bar\Lambda$ is an iterate of the shortest such chord
with the same initial point.
Using the index assumptions and
{Theorem~\ref{thm:brake-common-jump-standard},
which allows nonzero mean indices of either sign,}
we choose an interval of degrees in which the
iterates of each such shortest chord, together with their
reversals, contribute at most one generator.
The Morse inequalities in
Proposition~\ref{prop:rfh-morse-inequalities} therefore bound
the number of these shortest chords from below by the sum
of the dimensions of RFH in these degrees.
By Lemma~\ref{lem:real-half-period-quotient}, this gives the
same lower bound for the number of geometrically distinct
$\rho$-brake Reeb orbits in $Y$.

When $N_L>n$, this argument gives the bound $b(L)$ in
Theorem~\ref{thm:main-brake-total-betti}.
When $N_L=n$, it gives the lower bound $b(L)-2$.
{Comparing the two sides of
\eqref{eq:brake-prefix-parity} modulo two rules out equality
and gives the bound $b(L)-1$ in
Theorem~\ref{thm:main-brake-critical}.}
Under the additional mean index assumption,
Corollary~\ref{lem:brake-mean-density} yields the bound $b(L)$.

To prove Theorem~\ref{cor:sphere-pinching}, we combine
Rademacher's length estimate~\cite{Rademacher} with the
curvature bounds to obtain lower bounds for the two Maslov
indices and the mean index of Reeb chords.
For $a\ge1$, these estimates allow us to apply
Theorem~\ref{thm:main-brake-critical}.
For $a=0$, we instead use the RFH of each of the two connected
components of $\Lambda$, without passing to the quotient.

\section{Preliminaries}\label{sec:preliminaries}

\subsection{Reeb chords and cappings}
\label{subsec:geometric-conventions}

We use the notation and conventions of
Sections~\ref{subsec:conjectures} and~\ref{subsec:main-results}:
\[
 \pi:(Y^{2n+1},\xi)\longrightarrow(B^{2n},\omega),\qquad
 \xi=\ker\alpha_0,\qquad
 \Lambda=\Fix(\rho),\qquad L=\Fix(\tau).
\]
We normalize $\alpha_0$ so that its fiber orbits have minimal
period $1$, and denote the time-$1/2$ map of its Reeb flow by
$\delta=\phi_{\alpha_0}^{1/2}$.
The contact form is $\alpha=f\alpha_0$, where $f>0$ and
$f\circ\rho=f$. We impose $f\circ\delta=f$ when passing to
the quotient by $\delta$.

A \emph{Legendrian lift} of a Lagrangian $L\subset B$ is a Legendrian
submanifold $K\subset Y$ for which $\pi|_K:K\to L$ is a
covering map.
The fixed point set $\Lambda$ is a possibly disconnected
Legendrian lift of $L$.

Since $\rho^*\alpha=-\alpha$, the Reeb flow satisfies
$\rho\phi_\alpha^t=\phi_\alpha^{-t}\rho$.
Let $K$ be a submanifold of $\Fix(\rho)$.
For a Reeb chord $x:[0,T]\to Y$ of $K$,
reflection gives the smooth $\rho$-brake Reeb orbit
\begin{equation}\label{eq:brake-reflected-double}
 x^2(t)=
 \begin{cases}
  x(t),&0\le t\le T,\\
  \rho(x(2T-t)),&T\le t\le2T
 \end{cases}
\end{equation}
of period $2T$; see
\cite[Section~1.2]{KKK}.
If $x$ is relatively contractible in $(Y,K)$, it is homotopic,
with fixed endpoints, to a path $\eta$ in $K$.
Reflecting this homotopy gives a homotopy from $x^2$ to
$\eta*\eta^{-1}$, so $x^2$ is contractible.

If a prime closed Reeb orbit $\gamma$ of period $P$ starts at
a point of $\Fix(\rho)$, then
$\rho\gamma(t)=\gamma(-t)$. Since $P$ is the least period,
\begin{equation}\label{eq:two-brake-markings}
 \gamma(t)\in\Fix(\rho)
 \quad\Longleftrightarrow\quad 2t=0\pmod P.
\end{equation}

For $m\ge1$, the $m$th iterate $x^m:[0,mT]\to Y$
is obtained by repeated reflection at the endpoints of $x$;
see \cite[Section~4.1.3]{KKK}.

In the remainder of this subsection, let $\Lambda\subset Y$
be a possibly disconnected Legendrian lift of a positively monotone
Lagrangian $L\subset B$, with monotonicity constant $\lambda_L$
as in \eqref{eq:monotone}.
The connection identifies $\xi$ with $\pi^*TB$ and $T\Lambda$ with
the pullback of $TL$. Hence the relative Maslov index of a disk
$u:(D,\partial D)\to(Y,\Lambda)$ equals
\[
 \mu_L([\pi\circ u])
 =\lambda_L^{-1}\int_Du^*d\alpha_0
 =\lambda_L^{-1}\int_{\partial D}u^*\alpha_0=0,
\]
by monotonicity, Stokes' theorem, and horizontality of $\Lambda$.
Thus the relative Maslov class of $(\xi,T\Lambda)$ vanishes on
$\pi_2(Y,\Lambda)$; see \cite[Lemma~3.6]{BKK}.

For every sphere $v:S^2\to Y$, monotonicity, the identification
$\xi\cong\pi^*TB$, and Stokes' theorem give
\[
 \langle c_1(\xi),[v]\rangle
 =\langle c_1(TB),[\pi\circ v]\rangle
 =\frac{1}{2\lambda_L}\int_{S^2}v^*d\alpha_0
 =0.
\]
Hence $c_1(\xi)|_{\pi_2(Y)}=0$.

For a relatively contractible chord $x:[0,T]\to Y$, a \emph{relative
capping} is a map
\[
 u:\{z\in\mathbb C:|z|\le1,\ \operatorname{Im}z\ge0\}\longrightarrow Y
\]
whose semicircular boundary lies in $\Lambda$ and whose diameter
satisfies $u(1-2t)=x(Tt)$ for $0\le t\le1$. We use the
orientation convention of \cite[Section~3.2]{BKK}.
Choose a symplectic trivialization of $u^*(\xi,d\alpha|_\xi)$,
with fiberwise symplectic isomorphisms
$\Psi_z:\xi_{u(z)}\to\mathbb C^n$.
Since $d\alpha|_\xi=f\,d\alpha_0|_\xi$ with $f>0$, this choice of
symplectic form does not change the relative Maslov class.
Require that $\Psi_z(T_{u(z)}\Lambda)=\mathbb R^n$
along the semicircular boundary.
Since the Reeb flow preserves $d\alpha|_\xi$, its linearization
in this trivialization gives the symplectic path
\[
 \Phi_x(t)=\Psi_{1-2t/T}\circ
 \left.(D\phi_\alpha^t)\right|_{\xi_{x(0)}}\circ\Psi_1^{-1},
 \qquad 0\le t\le T,
\]
with $\Phi_x(0)=\Id$.

\subsection{Indices of Reeb chords}
\label{sec:index}

Throughout this subsection, $L$ is positively monotone, as in
Section~\ref{subsec:geometric-conventions}.
We briefly recall the Robbin--Salamon index; see
\cite[Section~2]{RS} and \cite[Section~4.1.1]{KKK}.
Let $\ell(t)$, $0\le t\le T$, be a smooth path of Lagrangian
subspaces of a symplectic vector space $(V,\Omega)$, and let
$F\subset V$ be a fixed Lagrangian subspace. A time $t_*$ is a
\emph{crossing} if $\ell(t_*)\cap F\ne\{0\}$. For
$v\in\ell(t_*)\cap F$, choose a smooth extension
$v(t)\in\ell(t)$ with $v(t_*)=v$. The \emph{crossing form} is
\[
 Q_{t_*}(v)=\Omega\bigl(v,\dot v(t_*)\bigr).
\]
It is independent of the choice of extension. A crossing is
\emph{regular} if $Q_{t_*}$ is nondegenerate. If all crossings are regular,
then
\[
 \mu_{\RS}(\ell,F)
 =\frac12\operatorname{sign}Q_0
  +\sum_{\substack{0<t<T\\ \ell(t)\cap F\ne\{0\}}}
       \operatorname{sign}Q_t
  +\frac12\operatorname{sign}Q_T.
\]
Here $\operatorname{sign}$ denotes the signature. At an endpoint
that is not a crossing, the corresponding term is zero.
For a general continuous path, the index is defined by homotopy
with fixed endpoints to a path with only regular crossings.

For a relatively contractible chord $x$, let $\Phi_x$ be the path
constructed in Section~\ref{subsec:geometric-conventions}.
Its Robbin--Salamon index is
\[
 \mu_{\RS}(x)=\mu_{\RS}(\Phi_x\mathbb R^n,\mathbb R^n).
\]
This index is independent of the adapted trivialization and,
since the relative Maslov class vanishes, of the capping
\cite[Theorems~2.3 and~3.1, and Remark~2.6]{RS}.
For a nondegenerate chord, $\Phi_x(0)=\Id$ and
$\Phi_x(T)\mathbb R^n\cap\mathbb R^n=\{0\}$ imply
$\mu_{\RS}(x)\in\frac n2+\mathbb Z$; see
\cite[Theorem~2.4]{RS}. Following \cite[Appendix~A]{BKK},
we use the integer \emph{Floer degree}
\begin{equation}\label{eq:degree}
 \mu(x)=\mu_{\RS}(x)+\frac n2.
\end{equation}

Suppose now that $(Y,\alpha,\rho)$ is a real contact manifold
and $\Lambda\subset\Fix(\rho)$.
Along $\Lambda$, the eigenspaces
\[
 E^+=\ker((D\rho)|_\xi-\Id)=T\Lambda,\qquad
 E^-=\ker((D\rho)|_\xi+\Id)
\]
are complementary Lagrangian subbundles of $\xi$. Following
\cite[Section~4.1.3]{KKK}, put
\[
 I=\operatorname{diag}(\Id_n,-\Id_n),\qquad
 L_0=\R^n\times\{0\}=\Fix(I),\qquad
 L_1=\{0\}\times\R^n=\Fix(-I).
\]
Choose the trivialization over the capping to identify $E^+$ and $E^-$
with $L_0$ and $L_1$ along the semicircular boundary.
The \emph{Maslov indices} of $x$ are defined by
\begin{equation}\label{eq:real-boundary-indices}
 \mu_I(x)=\mu_{\RS}(\Phi_xL_0,L_0),\qquad
 \mu_{-I}(x)=\mu_{\RS}(\Phi_xL_1,L_1).
\end{equation}
For $j=0,1$, we say that $\Phi_x$ is $L_j$-nondegenerate
if $\Phi_x(T)L_j\cap L_j=\{0\}$.
We use $\mu_I(x)$ and $\mu_{-I}(x)$ only when $\Phi_x$
is $L_0$- and $L_1$-nondegenerate, respectively.

For a nondegenerate chord $x$ of $\Lambda$, we have
$\mu_I(x)=\mu_{\RS}(x)$ and hence
$\mu(x)=\mu_I(x)+n/2$.
For degenerate chords, we use the Robbin--Salamon index rather than
the lower semicontinuous extensions in \cite[Definition~4.2]{KKK}.

For a symplectic path $\Phi:[0,T]\to\operatorname{Sp}(2n)$
with $\Phi(0)=\Id$ and $\det(\Phi(T)-\Id)\ne0$, our convention
for the Conley--Zehnder index is
\[
 \mu_{\mathrm{CZ}}(\Phi)
 =\mu_{\RS}\bigl(\operatorname{Gr}(\Phi),\Delta\bigr),
 \qquad
 \Delta=\{(v,v):v\in\mathbb R^{2n}\}.
\]
Here $\operatorname{Gr}(\Phi)$ denotes the path
$t\mapsto\{(v,\Phi(t)v):v\in\mathbb R^{2n}\}$.
The graph and diagonal are viewed as Lagrangian subspaces of
$(\mathbb R^{2n}\oplus\mathbb R^{2n},-\omega_0\oplus\omega_0)$,
where $\omega_0=\sum_{j=1}^n dq_j\wedge dp_j$; see
\cite[Remark~5.4]{RS}. For a nondegenerate contractible closed Reeb
orbit, we use a trivialization extending over a capping. The resulting
index is independent of the capping because
$c_1(\xi)|_{\pi_2(Y)}=0$, as shown in
Section~\ref{subsec:geometric-conventions}.

\begin{lemma}\label{lem:real-two-boundary-indices}
In the setting of Conjecture~\ref{conj:real-prequantization-bound},
suppose that $L$ is positively monotone. For every relatively
contractible chord of a Legendrian lift contained in
$\Fix(\rho)$, the indices
$\mu_{\RS}(\Phi_xL_j,L_j)$, $j=0,1$, are independent of the
relative capping and the adapted trivialization.
{For such a chord $x$ with relative capping $u$, a covering
$q:(Y,\alpha)\to(\bar Y,\bar\alpha)$ satisfying
$q^*\bar\alpha=\alpha$ and intertwining the anti-contact
involutions preserves both indices when $q\circ x$ is
capped by $q\circ u$.}
\end{lemma}
\begin{proof}
Choose a compatible complex structure $J$ on $\xi$ with
$D\rho\circ J=-J\circ D\rho$. Such a structure can be obtained
from a $\rho$-invariant metric.
Along the fixed set, $E^-=JE^+$. Thus
$e^{\theta J}E^+$, $0\le\theta\le\pi/2$, is a Lagrangian
homotopy from $E^+$ to $E^-$. By Section~\ref{subsec:geometric-conventions}, the relative
Maslov class of $(\xi,E^+)$ vanishes, and the homotopy gives the
same conclusion for $(\xi,E^-)$.
Since the semicircular boundary is an interval, choose a unitary
trivialization over the capping adapted to both eigenspaces along
this boundary. Changing the capping
changes either Maslov index by the corresponding relative
Maslov index, which is zero. The index is unchanged under a
change of adapted trivialization over a fixed capping, by the
naturality and homotopy properties in
\cite[Theorems~2.3 and~3.1]{RS}.

For the covering assertion, the differential $Dq$ identifies
the symplectic contact
distributions, the linearized Reeb flows, and the two eigenspace
bundles. Lifting or projecting a relative capping and its adapted
frame therefore gives identical matrix paths.
\end{proof}

\subsection{Index iteration}
\label{subsec:index-iteration-tools}

For a symplectic path
$\Phi:[0,T]\to\operatorname{Sp}(2n)$ with $\Phi(0)=\Id$,
let $\Phi^k:[0,kT]\to\operatorname{Sp}(2n)$ denote its
$k$th iterate, defined by reflection and concatenation
as in \cite[(4.5)--(4.6)]{KKK}.
In particular,
\[
 \Phi^2(t)=
 \begin{cases}
  \Phi(t),&0\le t\le T,\\
  I\Phi(2T-t)\Phi(T)^{-1}I\Phi(T),
    &T\le t\le2T.
 \end{cases}
\]
The path $\Phi^{2k}$ is obtained by iterating
$\Phi^2$ $k$ times, using the extension
$\Phi^2(t+2T)=\Phi^2(t)\Phi^2(2T)$; see
\cite[Section~4.1.3]{KKK}.

\begin{lemma}\label{lem:brake-capped-iteration}
In the setting of Conjecture~\ref{conj:real-prequantization-bound},
suppose that $L$ is positively monotone.
Let $K$ be a union of connected components of $\Fix(\rho)$,
and let $x:[0,T]\to Y$ be a
relatively contractible Reeb chord of $K$.
Suppose that every iterate of the associated $\rho$-brake
Reeb orbit $x^2$ is nondegenerate.
A relative capping of $x$, together with an adapted
symplectic trivialization, determines a path
$\Phi:[0,T]\to\operatorname{Sp}(2n)$ whose iterates
satisfy
\begin{equation}\label{eq:brake-capped-iteration}
 \mu_I(x^\ell)=\mu_{\RS}(\Phi^\ell L_0,L_0),\qquad
 \mu_{-I}(x^\ell)=\mu_{\RS}(\Phi^\ell L_1,L_1)
 \quad(\ell\ge1).
\end{equation}
Every $x^\ell$ is relatively contractible, and
\[
 \Phi^\ell(\ell T)L_j\cap L_j=\{0\}
 \qquad(j=0,1;\ \ell\ge1).
\]
\end{lemma}

\begin{proof}
Choose a symplectic trivialization over a relative capping
of $x$ that identifies the eigenspaces of $D\rho$ with
$L_0$ and $L_1$ along the semicircular boundary.
Reflect the capping using $\rho$ and the trivialization
using $D\rho$ and $I$.
{Repeated reflection and concatenation give relative
cappings of $x^\ell$ and trivializations in which the
linearized Reeb flow is represented by $\Phi^\ell$.}

The $\rho$-brake Reeb orbit obtained by reflecting $x^\ell$ is the
$\ell$th iterate of $x^2$.
Thus \cite[Proposition~4.3]{KKK} implies that $\Phi^\ell$
is both $L_0$- and $L_1$-nondegenerate.
Computing the Maslov indices in these trivializations
gives \eqref{eq:brake-capped-iteration}.
\end{proof}

For a chord as in Lemma~\ref{lem:brake-capped-iteration},
the capping conventions of Section~\ref{sec:index} and
\cite[Proposition~4.3]{KKK} give
\begin{equation}\label{eq:real-double-index}
 \mu_{\mathrm{CZ}}(x^2)=\mu_I(x)+\mu_{-I}(x).
\end{equation}

Suppose that all iterates of $\Phi^2$ are
nondegenerate.
The proof of \cite[Theorem~4.6]{KKK} uses only the
nondegeneracy of these iterates and gives
\begin{equation}\label{eq:brake-index-growth}
 \mu_{\RS}(\Phi^{\ell+1}L_0,L_0)
 -\mu_{\RS}(\Phi^\ell L_0,L_0)
 \ge\mu_{\RS}(\Phi L_0,L_0)-\frac n2
 \qquad(\ell\ge1).
\end{equation}

{Conjugation by the anti-symplectic involution $I$
preserves nondegeneracy and satisfies
$(I\Phi I)^\ell=I\Phi^\ell I$ for every $\ell\ge1$.}
Since $I$ preserves $L_0$ and $L_1$, it changes the sign
of each crossing form for both Maslov indices.
The crossing-form formula and homotopy invariance of the
Robbin--Salamon index therefore give
\begin{equation}\label{eq:brake-conjugated-indices}
 \mu_{\RS}\bigl((I\Phi I)^\ell L_j,L_j\bigr)
 =-\mu_{\RS}(\Phi^\ell L_j,L_j)
 \qquad(j=0,1;\ \ell\ge1).
\end{equation}
Under the same nondegeneracy assumption, applying
\eqref{eq:brake-index-growth} to $\Phi$ and $I\Phi I$
gives, for $\varepsilon\in\{-1,1\}$ and $\ell\ge1$,
\begin{equation}\label{eq:brake-signed-index-growth}
 \varepsilon\bigl(
  \mu_{\RS}(\Phi^{\ell+1}L_0,L_0)
  -\mu_{\RS}(\Phi^\ell L_0,L_0)\bigr)
 \ge \varepsilon\mu_{\RS}(\Phi L_0,L_0)-\frac n2.
\end{equation}

For a symplectic path
$\Gamma:[0,S]\to\operatorname{Sp}(2n)$ with
$\Gamma(0)=\Id$, define its iterates by extending
it according to
$\Gamma(t+S)=\Gamma(t)\Gamma(S)$.
If all these iterates are nondegenerate, then its \emph{mean
Conley--Zehnder index} is
\[
 \widehat\mu_{\mathrm{CZ}}(\Gamma)
 =\lim_{k\to\infty}
   \frac{\mu_{\mathrm{CZ}}(\Gamma|_{[0,kS]})}{k}.
\]
It satisfies
\begin{equation}\label{eq:mean-cz-estimate}
 \left|\widehat\mu_{\mathrm{CZ}}(\Gamma)
          -\mu_{\mathrm{CZ}}(\Gamma)\right|<n;
\end{equation}
see \cite[Section~3.1]{GGM}.

Let $\Phi:[0,T]\to\operatorname{Sp}(2n)$ satisfy
$\Phi(0)=\Id$, and suppose that every iterate is
both $L_0$- and $L_1$-nondegenerate.
Following \cite[Section~4.4]{KKK}, its \emph{mean
Robbin--Salamon index} relative to $L_0$ is
\[
 \widehat\mu_{\RS}(\Phi L_0,L_0)
 =\lim_{k\to\infty}
   \frac{\mu_{\RS}(\Phi^kL_0,L_0)}{k}
 =\frac12\widehat\mu_{\mathrm{CZ}}(\Phi^2).
\]
The existence of the limit and the last equality follow
from \cite[Proposition~4.1]{LZIteration}.
{Liu and Zhang~\cite{LZIteration} use the convention
$L_0=\{0\}\times\R^n$.}
After conjugating $\Phi$ by the symplectic map
$(q,p)\mapsto(p,-q)$, their $L_0$-index equals
$\mu_{\RS}(\Phi L_0,L_0)-n/2$ in our notation.
This constant shift does not affect the limit.

For a chord $x$ as in Lemma~\ref{lem:brake-capped-iteration},
let $\Phi_x$ be the associated symplectic path.
By \eqref{eq:brake-capped-iteration} and
\eqref{eq:brake-mean-definition},
\[
 \widehat\mu(x)
 =\widehat\mu_{\RS}(\Phi_xL_0,L_0)
 =\lim_{m\to\infty}\frac{\mu(x^m)}m.
\]
The last equality follows from \eqref{eq:degree}.
\par

{The common index jump theorem of Liu and Zhang
\cite[Theorem~1.5 and Remark~5.1]{LZIteration}
assumes positive mean indices.
Using the formulation in
\cite[Theorem~4.10 and Proposition~4.11]{KKK}
and the conjugation identity
\eqref{eq:brake-conjugated-indices},
we obtain the following version for nonzero mean indices
of either sign.}

\begin{theorem}
\label{thm:brake-common-jump-standard}
Let $\Phi_i:[0,1]\to\operatorname{Sp}(2n)$, $1\le i\le s$,
be symplectic paths with $\Phi_i(0)=\Id$ such that every
iterate is both $L_0$- and $L_1$-nondegenerate.
Suppose that
$\widehat\mu_{\RS}(\Phi_iL_0,L_0)\ne0$ for every $i$,
and put
\[
 \varepsilon_i
 =\operatorname{sign}\widehat\mu_{\RS}(\Phi_iL_0,L_0).
\]
For every integer $D\ge1$, there are arbitrarily large
$K\in D\mathbb N$ and positive integers $m_1,\ldots,m_s$
such that
\begin{equation}\label{eq:brake-common-jump-standard}
 \begin{aligned}
  \mu_{\RS}(\Phi_i^{2m_i-1}L_0,L_0)
    &=\varepsilon_i K-\mu_{\RS}(\Phi_iL_1,L_1),\\
  \mu_{\RS}(\Phi_i^{2m_i+1}L_0,L_0)
    &=\varepsilon_i K+\mu_{\RS}(\Phi_iL_0,L_0)
       \qquad(1\le i\le s).
 \end{aligned}
\end{equation}
\end{theorem}

\begin{proof}
For each $i$, set
\[
 \widetilde\Phi_i=
 \begin{cases}
  \Phi_i,&\varepsilon_i=1,\\
  I\Phi_i I,&\varepsilon_i=-1.
 \end{cases}
\]
By \eqref{eq:brake-conjugated-indices}, all iterates
of $\widetilde\Phi_i$ are both $L_0$- and
$L_1$-nondegenerate, and
\[
 \mu_{\RS}(\widetilde\Phi_i^\ell L_j,L_j)
 =\varepsilon_i\mu_{\RS}(\Phi_i^\ell L_j,L_j)
 \qquad(j=0,1;\ \ell\ge1).
\]
In particular, each $\widetilde\Phi_i$ has positive mean
index.

Fix $D\ge1$.
To obtain the divisibility condition, consider
\begin{equation}\label{eq:brake-auxiliary-hyperbolic}
 H(t)=
 \begin{pmatrix}
  \cosh(t)\Id_n&\sinh(t)\Id_n\\
  \sinh(t)\Id_n&\cosh(t)\Id_n
 \end{pmatrix},
 \qquad 0\le t\le1.
\end{equation}
Extending $H$ by reflection gives the same formula
for all $t\ge0$.
All its iterates are both $L_0$- and
$L_1$-nondegenerate, and
\[
 \mu_{\RS}(H|_{[0,\ell]}L_0,L_0)=\frac n2
 \qquad(\ell\ge1).
\]

Prepend to $H$ the symplectic loop that rotates one complex
coordinate through angle $2\pi Dt$, $0\le t\le1$, and fixes
the other coordinates.
Reparametrize the concatenation on $[0,1]$ and denote it
by $\Psi$.
The induced Lagrangian loop based at $L_0$ has Maslov
index $2D$.
Reflection reverses both the time direction and the
symplectic form, so the reflected loop has the same
Maslov index.
By additivity under concatenation, we obtain
\begin{equation}\label{eq:brake-auxiliary-index}
 \mu_{\RS}(\Psi^\ell L_0,L_0)
 =2D\ell+\frac n2
 \qquad(\ell\ge1).
\end{equation}
Inserting the loop does not change the endpoint of any
iterate.
Hence all iterates of $\Psi$ are both $L_0$- and
$L_1$-nondegenerate, and
\eqref{eq:brake-auxiliary-index} gives mean index $2D>0$.

Applying \cite[Theorem~4.10 and Proposition~4.11]{KKK}
to $\widetilde\Phi_1,\ldots,\widetilde\Phi_s,\Psi$
gives arbitrarily large common integers $K$.
Let $m_0$ be the integer corresponding to $\Psi$.
The identity for its $(2m_0+1)$st iterate gives
\[
 2D(2m_0+1)+\frac n2
 =K+2D+\frac n2.
\]
Thus $K=4Dm_0$, so $D\mid K$.
Multiplying the identities in \cite[Theorem~4.10]{KKK}
for $\widetilde\Phi_i$ by $\varepsilon_i$ gives
\eqref{eq:brake-common-jump-standard}.
\end{proof}

\subsection{Rabinowitz Floer homology and quantum homology}
\label{subsec:rfh-comparison}

Let $\pi:(Y^{2n+1},\xi)\to(B^{2n},\omega)$ be a
prequantization $S^1$-bundle with normalized connection contact
form $\alpha_0$.
Let $L\subset B$ be a closed, connected, embedded, positively
monotone Lagrangian submanifold, with $\dim L=n$ and finite
minimal Maslov number $N_L\ge2$.
Denote its monotonicity constant by $\lambda_L$.
Let $\Lambda\subset Y$ be a Legendrian lift of $L$ such that
$\pi|_\Lambda:\Lambda\to L$ is a diffeomorphism, and let
$\alpha=f\alpha_0$ with $f>0$.

In defining RFH below, we consider only paths that are
relatively contractible in $(Y,\Lambda)$.
We denote the resulting homology by $\RFH_*(\Lambda;\F)$.
When $N_L\ge3$, we follow \cite{BKK},
starting with the contact form $\alpha_0$.
The differential counts Floer strips in $\R\times Y$
with boundary on $\R\times\Lambda$.

For $N_L=2$, choose a Liouville filling $W$ of $Y$ and an
exact Lagrangian filling $C\subset W$ of $\Lambda$ satisfying
the hypotheses of Lemma~\ref{lem:relative-filling-comparison}.
The Floer strips then lie in the completion of $W$, with
boundary on the cylindrical completion of $C$.
The same lemma gives degree-preserving isomorphisms between
the RFH groups obtained from any two such pairs $(W,C)$.

Following \cite[Section~3.2]{BKK}, a \emph{generalized Reeb chord}
of $\Lambda$ is a path $x:[0,1]\to Y$ satisfying
\[
 \dot x(t)=T R_\alpha(x(t)),\qquad
 x(0),x(1)\in\Lambda,
\]
for some $T\in\R$.
The chord is positive when $T>0$, negative when $T<0$,
and constant when $T=0$.
{For $T>0$, $x$ is a Reeb chord as defined in
Section~\ref{subsec:conjectures}, reparametrized on $[0,1]$.
Its \emph{reversal} $x^{-1}(t):=x(1-t)$, $t\in[0,1]$,
is a generalized Reeb chord with parameter $-T$.}
For nonconstant nondegenerate generalized chords, we use
the grading \eqref{eq:degree}.
If $x$ is a positive nondegenerate chord, the reversal and
naturality properties of the Robbin--Salamon index
\cite[Section~4.1.1]{KKK} give
\begin{equation}\label{eq:reversal}
 \mu_{\RS}(x^{-1})=-\mu_{\RS}(x),\qquad
 \mu(x^{-1})=n-\mu(x).
\end{equation}

We use $\vee$-shaped Hamiltonians as in
\cite[Sections~3.3--3.4 and~5.1]{BKK}.
When $N_L\ge3$, take a Liouville cobordism together with
an exact Lagrangian cobordism from
$(\kappa\alpha_0,\Lambda)$ to $(\alpha,\Lambda)$,
where $\kappa>0$ is sufficiently small, as described
after \cite[Theorem~1.1]{BKK}.

{In the cobordism for $N_L\ge3$, or in the completion of $W$
for $N_L=2$, choose} a radial coordinate $R$ in a collar
of the contact hypersurface, normalized by $R=1$ on the
hypersurface.
In this collar, the Hamiltonian has the form $H=h(R)$
with $h''>0$.
A Hamiltonian chord lies at a constant radial level $R$
and projects to a generalized Reeb chord, with $T=h'(R)$.
Conversely, if $T$ lies in the range of $h'$, the generalized
Reeb chord determines a unique Hamiltonian chord in this collar.
The constant Hamiltonian chords occur where $h'(R)=0$
and form a Morse--Bott critical manifold naturally
identified with $\Lambda$.

For each fixed finite action window, we choose the
Hamiltonians so that all generators outside this collar
have action below the lower cutoff.
{In the completion of $W$, the generators in the action window
are the type~III chords of \cite[Section~5.2]{BKK}.}
In the Morse--Bott description, the differential counts
rigid configurations of Floer strips with Lagrangian
boundary conditions, joined by gradient trajectories of
auxiliary Morse functions.
It has degree $-1$ and does not increase action.

The action spectrum of an admissible Hamiltonian $H$ is
the set of action values of its relatively contractible
Hamiltonian chords.
For $a<b$ outside this spectrum, let
$\mathrm{CF}_*^{<a}(H)$ denote the subcomplex generated
by elements of action less than $a$.
Set
\[
 \mathrm{CF}_*^{(a,b)}(H)
 =\mathrm{CF}_*^{<b}(H)/\mathrm{CF}_*^{<a}(H),\qquad
 \mathrm{HF}_j^{(a,b)}(H)
 =H_j(\mathrm{CF}_*^{(a,b)}(H)),
 \quad j\in\Z.
\]
We omit the almost complex structure and other auxiliary
choices from the notation.

Let $x$ be an isolated nonconstant relatively contractible
generalized Reeb chord, and consider the corresponding
Hamiltonian chord in the collar.
Choose a sufficiently small isolating neighborhood
$\mathcal U$ of this chord in the space of paths contained
in the collar.
After a sufficiently small regular perturbation of the
Floer data, the \emph{local Floer complex} is generated by the
Hamiltonian chords in $\mathcal U$.
Its differential counts rigid Floer strips $u$ satisfying
$u(s,\cdot)\in\mathcal U$ for every $s\in\R$.
The resulting homology is independent of the sufficiently
small regular perturbation and the choice of $\mathcal U$;
see \cite[Section~4]{MorenoLimoge}.
We denote it by $\mathrm{HF}_*^{\mathrm{loc}}(x;\F)$,
with the grading determined by relative cappings.

{The graded local Floer homology
$\mathrm{HF}_*^{\mathrm{loc}}(x;\F)$ is independent of
the choice of $h$.}
Fix $x$ and vary $h$ through strictly convex functions for
which the corresponding Hamiltonian chord lies in the collar.
This chord remains isolated throughout the deformation,
so local continuation gives degree-preserving isomorphisms.

If $x$ is nondegenerate, its local Floer complex has one
generator in degree $\mu(x)$ and zero differential.
Hence
\begin{equation}\label{eq:local-floer-nondegenerate}
 \mathrm{HF}_j^{\mathrm{loc}}(x;\F)
 \cong
 \begin{cases}
  \F, & j=\mu(x),\\
  0, & j\ne\mu(x).
 \end{cases}
\end{equation}

{We write $\mathrm{HF}_*^{\mathrm{loc}}(\Lambda;\F)$ for the
local Floer homology of the Morse--Bott family of constant
Hamiltonian chords.}
Choose an auxiliary Morse function $g$ on $\Lambda$.
{The local Floer complex is identified with a Morse complex,
with our grading $\mu(p)=n-\ind_g(p)$; see the
Morse--Bott calculation in \cite[Section~3.7]{KKK}.
Its degree-$j$ homology is $H^{n-j}(\Lambda;\F)$.}
Poincar\'e duality over $\F$ therefore gives
\begin{equation}\label{eq:local-floer-constant}
 \dim_\F\mathrm{HF}_j^{\mathrm{loc}}(\Lambda;\F)
 =b_j(\Lambda)=b_j(L).
\end{equation}

The \emph{action spectrum} of relatively contractible
generalized Reeb chords is
\begin{equation}\label{eq:reeb-action-spectrum}
 \operatorname{Spec}(\alpha,\Lambda)
 :=\{0\}\cup
 \left\{\pm\int_x\alpha:
       x\in\cC^0_+(\alpha,\Lambda)\right\}.
\end{equation}
For $a<b$ with $a,b\notin\operatorname{Spec}(\alpha,\Lambda)$,
the filtered homology for $\alpha$ is defined by
\[
 \RFH_j^{(a,b)}(\Lambda;\F)
 =\varinjlim_H\mathrm{HF}_j^{(a,b)}(H),
\]
where the direct limit is taken over {the admissible
$\vee$-shaped Hamiltonians}, with respect to the
continuation maps; see \cite[Section~3.4.2]{BKK}.

The full homology is defined by
\begin{equation}\label{eq:limits}
 \RFH_j(\Lambda;\F)=\varinjlim_{b\to+\infty}
          \varprojlim_{a\to-\infty}\RFH_j^{(a,b)}(\Lambda;\F).
\end{equation}
We omit $\Lambda$ and $\F$ from the notation when they are
clear from the context.
The short exact sequences associated with the action
filtration induce long exact sequences in filtered Floer
homology.
The maps in these sequences commute with continuation maps;
see \cite[Section~8.3]{CO}.

\begin{lemma}\label{lem:finiteaction}
Assume that every positive relatively contractible chord
of $\Lambda$ is nondegenerate.
Then, for every $A>0$, there are only finitely many chords
in $\cC^0_+(\alpha,\Lambda)$ with $T\le A$.
\end{lemma}

\begin{proof}
Since $R_\alpha$ is nowhere tangent to the compact submanifold
$\Lambda$, there exists $\epsilon>0$ such that $T\ge\epsilon$
for every positive chord.

The pairs $(p,T)\in\Lambda\times[\epsilon,A]$ satisfying
$\phi_\alpha^T(p)\in\Lambda$ form a compact set.
Relative homotopy classes are locally constant in the space
of paths with endpoints on $\Lambda$, so the subset
corresponding to relatively contractible chords is closed.
At every such pair, the map
$(p,T)\mapsto\phi_\alpha^T(p)$ is transverse to $\Lambda$:
the time direction is $R_\alpha$, and
\eqref{eq:real-chord-nondegeneracy} gives transversality
in $\xi$.
This subset is therefore discrete and compact, hence finite.
\end{proof}

{Assume the nondegeneracy hypothesis of
Lemma~\ref{lem:finiteaction}.}
For $H=h(R)$, let $R(T)$ be determined by $h'(R(T))=T$
in the collar.
The corresponding Hamiltonian chord has action
$\mathcal A_H(T)=R(T)T-h(R(T))$, and
\begin{equation}\label{eq:action-convergence-v2}
 \mathcal A_H(T)-\mathcal A_H(0)=\int_0^T R(s)\,ds.
\end{equation}
Choose a cofinal family $H_\eta$, with $h_\eta''>0$ in the
collar, such that $\mathcal A_{H_\eta}(0)\to0$ as $\eta\to0$
and, for every fixed $M>0$, the solutions of
$h_\eta'(R_\eta(T))=T$ satisfy
\[
 e^{-\eta}<R_\eta(T)<e^\eta
 \qquad(|T|\le M)
\]
for all sufficiently small $\eta>0$.
By \eqref{eq:action-convergence-v2},
$\mathcal A_{H_\eta}(T)\to T$ uniformly for $T$ in bounded
intervals.

Fix cutoffs $a_0<\cdots<a_q$ with
$a_i\notin\operatorname{Spec}(\alpha,\Lambda)$
for $0\le i\le q$.
By Lemma~\ref{lem:finiteaction} and
\eqref{eq:reeb-action-spectrum}, each cutoff has positive
distance from this spectrum.
For all sufficiently small $\eta>0$, we have
\[
 \mathcal A_{H_\eta}(T)<a_i
 \quad\Longleftrightarrow\quad
 T<a_i,
 \qquad 0\le i\le q,
\]
for every relatively contractible generalized Reeb chord
represented by a Hamiltonian chord in the collar.
Indeed, choose $M>1+\max_i|a_i|$.
For $|T|\le M$, the equivalence follows from uniform
convergence and the positive distance of the cutoffs from
the spectrum.
For $|T|>M$, it follows from
$\mathcal A_{H_\eta}'(T)=R_\eta(T)>0$ and
\[
 \mathcal A_{H_\eta}(-M)<a_0,\qquad
 \mathcal A_{H_\eta}(M)>a_q.
\]
These equivalences persist after sufficiently small auxiliary
perturbations, including the Morse perturbation of the
constant chords.

For fixed $A<B$ with
$A,B\notin\operatorname{Spec}(\alpha,\Lambda)$,
the continuation maps in the chosen cofinal family induce
isomorphisms on filtered homology for all sufficiently small
$\eta>0$.
The stabilization described in \cite[p.~1999]{CO}
applies with Lagrangian boundary conditions as in
\cite[Section~8.3]{CO}.
{For compactness, we use the symplectization construction of
\cite{BKK} when $N_L\ge3$, and the filling construction in the
proof of Lemma~\ref{lem:relative-filling-comparison} when $N_L=2$.}

For such $\eta$, the finite-dimensional complex
$\mathrm{CF}_*^{(A,B)}(H_\eta)$ computes $\RFH_*^{(A,B)}$.
It has one generator for each nonconstant relatively
contractible generalized Reeb chord with $T\in(A,B)$,
together with one generator for each critical
point of $g$ when $0\in(A,B)$.
For any fixed finite collection of action windows, these
complexes can be defined using the same Hamiltonian.
The corresponding long exact sequences can therefore be
formed before passing to the direct limit over Hamiltonians.

Fix such a Hamiltonian $H$ and an action window $(A,B)$
with $A<0<B$.
Choose a small interval
$(-\varepsilon,\varepsilon)\subset(A,B)$ containing exactly
the generators corresponding to constant chords.
In the collar, choose precompact neighborhoods
$\mathcal V\Subset\mathcal U$ of the family of constant
Hamiltonian chords such that no chord of the unperturbed
Hamiltonian meets $\overline{\mathcal U}\setminus\mathcal V$.
For sufficiently small perturbations, the generators in this
action interval lie in $\mathcal V$.
The argument in the proof of \cite[Lemma~4.2]{MorenoLimoge}
gives a constant $e_0>0$, uniform for these perturbations,
such that every Floer strip joining these generators and
leaving $\mathcal U$ has energy at least $e_0$.
{Otherwise, there would be Floer data $(J_\nu,H_\nu)$
converging to the unperturbed data and strips $u_\nu$
leaving $\mathcal U$ with $E(u_\nu)\to0$.
Translate each strip in the $s$-variable so that
$u_\nu(s,t)\in\mathcal U$ for all $s<0$ and $t\in[0,1]$,
while $u_\nu(0,t_\nu)\in\partial\mathcal U$ for some
$t_\nu\in[0,1]$.
After passing to a subsequence, $t_\nu\to t_\infty$, and
elliptic estimates for the Floer equation with Lagrangian
boundary conditions give local convergence to a zero-energy
solution $u_\infty$ of the unperturbed equation.
Thus $u_\infty$ is independent of $s$ and defines a Hamiltonian
chord $x_\infty$ with
$x_\infty(t_\infty)\in\partial\mathcal U$,
contrary to the choice of $\mathcal U$.}

Choose the perturbations, including the Morse perturbation,
sufficiently small that the action differences between these
generators are less than $e_0$.
The energy identity then implies that all strips joining
them remain in $\mathcal U$.
Thus $\mathrm{CF}_*^{(-\varepsilon,\varepsilon)}(H)$ computes
$\mathrm{HF}_*^{\mathrm{loc}}(\Lambda;\F)$.
Its degree-$j$ homology has dimension $b_j(L)$ by
\eqref{eq:local-floer-constant}.

{The Lagrangian quantum homology $\QH_*(L;\F)$ is computed
by the pearl complex over $\F[t,t^{-1}]$, where $|t|=-N_L$;
see \cite{BC,BKK}.}
Its differential counts configurations of $J$-holomorphic
disks with boundary on $L$, joined by gradient trajectories
of an auxiliary Morse function.
Powers of $t$ record the total Maslov index in multiples
of $N_L$.

For $N_L\ge3$, \cite[Theorem~1.1]{BKK} gives a
degree-preserving isomorphism
\begin{equation}\label{eq:rfhqh}
 \RFH_j(\Lambda;\F)\cong\QH_j(L;\F).
\end{equation}
When $L$ is wide, \eqref{eq:wide-module} gives
\begin{equation}\label{eq:wide-grading}
 \QH_j(L;\F)\cong\bigoplus_{k\in\Z}H_{j-kN_L}(L;\F).
\end{equation}
Consequently, for every $J\in\Z$,
\[
 \sum_{j=J}^{J+N_L-1}\dim_\F\QH_j(L;\F)=b(L).
\]

For the connection contact form $\alpha_0$, put
$a_L=\lambda_LN_L$.
This is the positive generator of
$\omega(\pi_2(B,L))\subset\R$.
Since $\Lambda$ is a horizontal section over $L$, the
connection has trivial holonomy along the boundary of
every relative disk.
Hence $\omega(A)\in\Z$ for every $A\in\pi_2(B,L)$, and
$a_L$ is a positive integer.
The relatively contractible generalized Reeb chords of
$\alpha_0$ form families with $T=ka_L$, $k\in\Z$,
each parametrized by $\Lambda$.

Using the diffeomorphism $\pi|_\Lambda$, regard $g$ as a
Morse function on $L$.
Write $x_k$ for a chord in the family with $T=ka_L$ and
$p_k$ for the generator corresponding to $p\in\Crit(g)$.
In the notation of \cite{BKK}, $a_L=m_L$, and $p_k$
corresponds to a chord with relative winding number $ka_L$.
By \cite[Propositions~3.4 and~3.7, formula~(3.11)]{BKK},
\begin{equation}\label{eq:mbdegree}
 \mu_{\RS}(x_k)=kN_L,\qquad
 \mu(p_k)=kN_L+n-\ind_g(p),
 \qquad p\in\Crit(g).
\end{equation}
Poincar\'e duality over $\F$ identifies the degree-$j$
local homology of the family with $T=ka_L$ with
$H_{j-kN_L}(L;\F)$.

\section{Auxiliary results}
\label{sec:auxiliary-results}

\subsection{\texorpdfstring{The quotient by $\delta$}{The quotient by delta}}
\label{subsec:brake-quotient}

\begin{lemma}\label{lem:real-half-period-quotient}
In the setting of Conjecture~\ref{conj:real-prequantization-bound},
suppose $f\circ\delta=f$, and let
$q:Y\to\bar Y=Y/\langle\delta\rangle$. Then $\bar Y$ is a
prequantization $S^1$-bundle over $(B,2\omega)$ with normalized
connection contact form $\bar\alpha_0$ satisfying
\begin{equation}\label{eq:real-quotient-forms}
 q^*\bar\alpha_0=2\alpha_0,\qquad
 q^*\bar\alpha=\alpha,\qquad
 \bar\alpha=\frac{\bar f}{2}\bar\alpha_0,
 \quad \bar f\circ q=f.
\end{equation}
The involution $\rho$ descends to an involution $\bar\rho$ with
$\bar\rho^*\bar\alpha_0=-\bar\alpha_0$ and
$\bar\rho^*\bar\alpha=-\bar\alpha$. The submanifold
$\bar\Lambda=q(\Lambda)$ is a Legendrian lift
that projects diffeomorphically onto $L$.
It is a union of connected components of $\Fix(\bar\rho)$
and satisfies $q^{-1}(\bar\Lambda)=\Lambda$.

Let $\bar\gamma$ be a prime closed Reeb orbit of $\bar\alpha$
meeting $\bar\Lambda$, and set
$s_{\bar\gamma}=\#(\bar\gamma\cap\bar\Lambda)\in\{1,2\}$.
Then $q^{-1}(\bar\gamma)$ consists of exactly
$s_{\bar\gamma}$ geometrically distinct $\rho$-brake Reeb orbits.
\end{lemma}

\begin{proof}
The identity
$\rho\phi_{\alpha_0}^t=\phi_{\alpha_0}^{-t}\rho$ implies
that $\rho$ commutes with $\delta$ and hence descends to
$\bar\rho$ on $\bar Y$. On each fiber over $L$,
$\rho$ reverses the circle action and has two fixed points
exchanged by $\delta$.
Thus $\bar\Lambda$ projects diffeomorphically onto $L$, and
$q^{-1}(\bar\Lambda)=\Lambda$.
The fixed set $\Lambda$ is horizontal because
$\rho^*\alpha_0=-\alpha_0$, so its image $\bar\Lambda$ is
Legendrian.
The submanifold $\bar\Lambda$ lies in $\Fix(\bar\rho)$.
It is closed and has the same dimension as $\Fix(\bar\rho)$,
so it is a union of its connected components.

The form $\alpha_0$ descends to a contact form on $\bar Y$
whose fiber orbits have minimal period $1/2$.
Multiplying this descended form by $2$ gives the normalized
connection contact form $\bar\alpha_0$ for the prequantization over
$(B,2\omega)$.
The assumption $f\circ\delta=f$ gives the descended contact
form $\bar\alpha$ in \eqref{eq:real-quotient-forms}.
The identity $\rho^*\alpha=-\alpha$ then implies
$\bar\rho^*\bar\alpha=-\bar\alpha$.

Since $q$ is a finite covering and intertwines the Reeb flows,
each connected component of $q^{-1}(\bar\gamma)$ is the image
of a prime closed Reeb orbit.
Each such orbit meets $\Lambda$ and is therefore $\rho$-brake.
Because $q$ is two-to-one and
$q^{-1}(\bar\Lambda)=\Lambda$, the set
$q^{-1}(\bar\gamma)\cap\Lambda$ contains exactly
$2s_{\bar\gamma}$ points.
By \eqref{eq:two-brake-markings}, each lifted prime Reeb orbit
contains exactly two points of $\Lambda$.
Hence the number of geometrically distinct lifted
$\rho$-brake Reeb orbits is $s_{\bar\gamma}$.
\end{proof}

If $L$ is positively monotone with constant $\lambda_L$ for
$\omega$, its monotonicity constant for $2\omega$ is $2\lambda_L$.
The minimal Maslov number and wideness are unchanged
by this rescaling.

Since the half-disk is simply connected, every relative capping
in $(\bar Y,\bar\Lambda)$ lifts to $Y$ once a lift of the chord's
initial point is chosen. The identity
$q^{-1}(\bar\Lambda)=\Lambda$ ensures that this lift is a relative
capping in $(Y,\Lambda)$. Conversely, relative cappings in
$(Y,\Lambda)$ project to relative cappings in
$(\bar Y,\bar\Lambda)$. If $L$ is positively monotone,
Lemma~\ref{lem:real-two-boundary-indices} shows that the two
Maslov indices agree for a chord and its lift.
Their Floer degrees therefore agree as well.

Suppose that every prime $\rho$-brake Reeb orbit in $Y$
carrying a positive relatively contractible chord of
$\Lambda$ is strongly nondegenerate.
Let $\bar\gamma$ be a prime closed Reeb orbit in $\bar Y$
carrying a positive relatively contractible chord of
$\bar\Lambda$. Choose a prime Reeb orbit in $Y$ covering
$\bar\gamma$, and let $e\in\{1,2\}$ be its covering degree.
By the preceding paragraph, the chosen orbit carries a positive
relatively contractible chord of $\Lambda$. It is therefore
a $\rho$-brake Reeb orbit and is strongly nondegenerate by hypothesis.
Let $P$ be the linearized Poincar\'e return map of $\bar\gamma$.
The linearized Poincar\'e return map of the chosen orbit is
conjugate to $P^e$. If $P^m$ had eigenvalue $1$ for some $m\ge1$, then
$(P^e)^m=(P^m)^e$ would also have eigenvalue $1$, contrary to
the strong nondegeneracy of the chosen orbit.
Thus $\bar\gamma$ is strongly nondegenerate. In particular,
for every such chord $x$, the associated $\bar\rho$-brake Reeb
orbit $x^2$ and all its iterates are nondegenerate.

\begin{lemma}\label{lem:brake-admissible-iterates}
Let $(Y,\alpha,\rho)$ be a real contact manifold, and let $K$
be a union of connected components of $\Fix(\rho)$.
Let $\gamma$ be a prime $\rho$-brake Reeb orbit of period $T$
that carries a positive relatively contractible chord of $K$.
For each $p\in\gamma\cap K$, there exists a shortest positive
relatively contractible chord $x_p$ of $K$ starting at $p$.
The positive relatively contractible chords of $K$ starting
at $p$ are precisely the iterates of $x_p$.
\end{lemma}

\begin{proof}
\leavevmode The existence of a positive relatively contractible
chord of $K$ implies that some iterate of $\gamma$ is
contractible, by the reflection argument in
Section~\ref{subsec:geometric-conventions}. Fix
$p\in\gamma\cap K$ and regard $\gamma$ as a $T$-periodic map
$\R\to Y$ with $\gamma(0)=p$. The class
$[\gamma]\in\pi_1(Y,p)$ therefore has finite order $r\ge1$;
this order is independent of the base point on $\gamma$.

For $k\ge1$, put $x_k=\gamma|_{[0,kT/2]}$.
{By \eqref{eq:two-brake-markings}, the positive chords of
$K$ starting at $p$ are precisely those paths $x_k$ for
which $\gamma(kT/2)\in K$.}
Whenever $x_k$ is a relatively contractible chord of $K$,
reflection of a relative capping gives a capping of $\gamma^k$.
Hence $r$ divides every such $k$.

The path $x_{2r}=\gamma^r$ is a contractible loop based at $p$.
The path $x_{k+2r}$ is obtained from $x_k$ by following the orbit
for an additional time $rT$. This additional segment is a
contractible loop, so $x_{k+2r}$ is homotopic to $x_k$ with
endpoints fixed.
Every $x_{2\ell r}$, $\ell\ge1$, is therefore a relatively
contractible chord of $K$.
If $k$ is an odd multiple of $r$, then $x_k$ has the same
endpoint as $x_r$ and is homotopic to it with endpoints fixed.

If $x_r$ is a relatively contractible chord of $K$,
then $x_k$ is such a chord exactly when $k\in r\mathbb N$.
In this case, the shortest chord is $x_p=x_r$.
Otherwise, $x_k$ is a relatively contractible chord of $K$
exactly when $k\in2r\mathbb N$, and the shortest chord is
$x_p=x_{2r}$.
Reflection at the endpoints of $x_p$ continues it along the
same Reeb trajectory, so its iterates are exactly the
chords just listed.
\end{proof}

Every negative chord is the reversal of a positive one,
and reversal preserves relative contractibility.
The preceding lemma therefore shows that every nonconstant
relatively contractible generalized chord of $K$ has
the form $x_p^m$ or $(x_p^m)^{-1}$, with $m\ge1$.
Although $x_p$ is the shortest positive relatively contractible
chord starting at $p$, the associated $\rho$-brake Reeb
orbit $x_p^2$ need not be the shortest contractible iterate of the
underlying prime orbit.

Let $\bar\gamma$ be a prime orbit in $\bar Y$
carrying a positive relatively contractible chord of
$\bar\Lambda$.
{By Lemma~\ref{lem:brake-admissible-iterates}, the positive
relatively contractible chords on $\bar\gamma$ are precisely
$x_p^m$, where $p\in\bar\gamma\cap\bar\Lambda$ and $m\ge1$.
By Lemma~\ref{lem:real-half-period-quotient}, the number
$s_{\bar\gamma}$ of these initial points equals the number
of geometrically distinct $\rho$-brake Reeb orbits in
$q^{-1}(\bar\gamma)$.}

\subsection{Comparison with quantum homology}
\label{subsec:relative-filling-comparison}

\begin{lemma}\label{lem:relative-filling-comparison}
In the setting of Section~\ref{subsec:rfh-comparison},
suppose that $\dim L=n\ge1$ and $N_L=2$.
Let $W$ be a Liouville filling of $Y$, and let
$C\subset W$ be an exact Lagrangian filling of $\Lambda$.
Assume that
\begin{equation}\label{ass:relative-injectivity}
 \pi_1(Y,\Lambda)\longrightarrow\pi_1(W,C)
 \quad\text{is injective},
\end{equation}
and that $\mu_C|_{\pi_2(W,C)}=0$.
Then there is a degree-preserving isomorphism
\[
 \RFH_*(\Lambda;\F)\cong\QH_*(L;\F),
\]
where RFH is defined using $(W,C)$.
The grading on RFH agrees with that determined by
relative cappings in $(Y,\Lambda)$.
\end{lemma}

\begin{proof}
\leavevmode We adapt the neck-stretching and comparison
arguments in \cite[Propositions~5.1(a) and~5.2(a)]{BKK}, replacing
the assumption that $\pi_1(Y)\to\pi_1(W)$ is injective with
\eqref{ass:relative-injectivity}.

Since $W$ and $C$ are exact, the Floer construction in the
completion of $W$ gives the filtered complexes and continuation
maps described in Section~\ref{subsec:rfh-comparison},
with $\partial^2=0$.
The assumption $\mu_C|_{\pi_2(W,C)}=0$ makes the integer
grading independent of the relative capping.
It also implies $c_1(TW)|_{\pi_2(W)}=0$, since the Maslov
index of a relative class induced by a sphere is twice
its first Chern number.
By \eqref{ass:relative-injectivity}, a chord in $Y$ with
endpoints on $\Lambda$ is relatively contractible in $(W,C)$
if and only if it is relatively contractible in $(Y,\Lambda)$.
{In each fixed action window, the symplectization and filling
constructions therefore have the same generators and the same
grading, computed using relative cappings in $(Y,\Lambda)$.}

We first consider $\alpha_0$.
{Fix a finite action window and finitely many admissible
Hamiltonians from Section~\ref{subsec:rfh-comparison}, with
monotone homotopies between them as in
\cite[Section~5.1]{BKK}.
Choose regular Morse data and an almost complex structure
on the symplectization as in
\cite[Propositions~3.11 and~3.12(a)]{BKK}, and keep both fixed
during the Hamiltonian homotopies.}
Extend the almost complex structure to the completion of $W$
and stretch a neck along a copy of $Y$ below the chosen
Hamiltonian chords, where the Hamiltonians are constant.
We claim that, for a sufficiently long neck, the strips
counted by the differential and continuation maps remain
above this hypersurface.

{Suppose instead that, for arbitrarily long necks, a strip
counted by the differential or a continuation map crosses
the chosen copy of $Y$.
A neck-stretching limit then has at least one negative
puncture on its top level.}
The negative punctures are asymptotic to positive Reeb
orbits or chords along the circle fibers.
For an interior puncture, a small circle around it corresponds
to a circle in the domain of an approximating strip.
This circle bounds a disk, and its image is freely homotopic
in $W$ to the asymptotic orbit.
Hence the orbit is contractible in $W$.
For a boundary puncture, cut the approximating strip along
the corresponding arc and cap the asymptotic generators on
one side of the cut.
The resulting disk has boundary on $C$ except along this arc,
which is homotopic to the asymptotic chord with endpoints
on $C$.
Thus the chord is relatively contractible in $(W,C)$.

Move each asymptotic orbit at an interior puncture through
fiber orbits to a fiber over $L$, with base point in $\Lambda$.
As a chord of $\Lambda$, it is trivial in $\pi_1(W,C)$
and therefore in $\pi_1(Y,\Lambda)$ by
\eqref{ass:relative-injectivity}.
The same injectivity assumption applies to the asymptotic
chords at boundary punctures.
Let $r\ge1$ be the total number of negative punctures.
Since $\pi|_\Lambda:\Lambda\to L$ is a diffeomorphism,
the asymptotic fiber orbits and chords have positive integer
winding numbers $k_h$, $1\le h\le r$.
Their relative contractibility gives $k_h\in a_L\Z$ by
\cite[Proposition~3.4]{BKK}.
Set
\[
 \nu_h:=\frac{k_h}{\lambda_L}
       =\frac{k_h}{a_L}N_L\ge2.
\]

Let $p_k$ and $q_l$ be the input and output generators,
with $p,q\in\Crit(g)$ and $k,l\in\Z$, and put
$d=\mu(p_k)-\mu(q_l)$.
Thus $d=1$ for the differential and $d=0$ for continuation maps.
Let $\mathbf A=(A_1,\ldots,A_N)$ be the disk classes obtained
by projecting the top-level configuration to $B$, and write
\[
 \mu_L(\mathbf A)=\sum_{j=1}^N\mu_L(A_j),\qquad
 \omega(\mathbf A)=\sum_{j=1}^N\omega(A_j).
\]
The input and output have fiber winding numbers $ka_L$ and
$la_L$, respectively.
Stokes' theorem gives
\[
 \omega(\mathbf A)=a_L(k-l)-\sum_{h=1}^r k_h.
\]
Dividing by $\lambda_L$ and using
\eqref{eq:monotone} and \eqref{eq:mbdegree}, we obtain,
as in \cite[(3.32)]{BKK},
\begin{equation}\label{eq:relative-filling-index-balance}
 d=\ind_g(q)-\ind_g(p)+\mu_L(\mathbf A)
          +\sum_{h=1}^r\nu_h.
\end{equation}
This computation does not require the asymptotic orbits at
interior punctures to be contractible in $Y$.
Since all projected disks have nonnegative area,
$\mu_L(\mathbf A)\ge0$.

If $n\le2$, then
\[
 0\le\mu_L(\mathbf A)
   \le d+n-\sum_{h=1}^r\nu_h\le1.
\]
Since $N_L=2$, this implies $\mu_L(\mathbf A)=0$.
All projected disks are therefore constant, and the negative
gradient trajectories of $g$ form a chain from $q$ to $p$.
The Morse--Smale condition implies
$\ind_g(q)\ge\ind_g(p)$.
Hence \eqref{eq:relative-filling-index-balance} gives
$d\ge\sum_{h=1}^r\nu_h\ge2$, a contradiction.

Suppose that $n\ge3$.
As in the proof of \cite[Proposition~3.12(b)]{BKK},
take an underlying simple chain of pearls and denote its
disk classes by $\mathbf A^*$.
Its total area does not exceed that of $\mathbf A$, so
\[
 \mu_L(\mathbf A)-\mu_L(\mathbf A^*)\ge0.
\]
Let $N^*$ be the number of disks in this chain and $N_0^*$
the number of constant disks.
The automorphisms fixing the two boundary marked points
act freely on each nonconstant simple disk.
Reparametrizing these disks independently therefore gives
a free $\R^{N^*-N_0^*}$-action on the moduli space of
parametrized simple chains.
The dimension formula in
\cite[Proposition~2.4(a) and the proof of Proposition~3.12(b)]{BKK}
gives
\[
 \ind_g(q)-\ind_g(p)+\mu_L(\mathbf A^*)+N^*-1
       \ge N^*-N_0^*.
\]
Combining this with
\eqref{eq:relative-filling-index-balance}, we obtain
\[
 d\ge1-N_0^*+\mu_L(\mathbf A)-\mu_L(\mathbf A^*)
              +\sum_{h=1}^r\nu_h.
\]

For the differential, each constant disk retained in the
simple chain comes from a distinct Floer component with
constant projection.
If such a component has no negative punctures, its asymptotic
chords have the same fiber winding number.
Since they lie on the strictly convex part of the same
radial Hamiltonian, they have the same action.
The energy identity then makes the component a trivial strip,
which is omitted from the cascade.
Thus $N_0^*\le r$.
Using $d=1$, we obtain, as in \cite[(3.34)]{BKK},
\[
 1\ge1+\mu_L(\mathbf A)-\mu_L(\mathbf A^*)
          +\sum_{h=1}^r(\nu_h-1)\ge2,
\]
a contradiction.

For continuation, the continuation strip may have constant
projection without negative punctures.
Every other component with constant projection has at least
one negative puncture, so $N_0^*\le r+1$.
Using $d=0$, we obtain, as in the proof of
\cite[Proposition~3.18(a)]{BKK},
\[
 0\ge\mu_L(\mathbf A)-\mu_L(\mathbf A^*)
          +\sum_{h=1}^r(\nu_h-1)\ge1,
\]
again a contradiction.

For any fixed finite collection of action windows,
{admissible Hamiltonians, and monotone homotopies with
fixed Morse data and almost complex structure}, the
differentials and continuation maps in the completion of $W$
therefore agree with those in the symplectization once the
neck is sufficiently long.
The projection and lifting argument in
\cite[Section~4.3]{BKK} identifies the Floer differential
for $\alpha_0$ with the pearl differential of $L$ on the
corresponding generators.
By \cite[Proposition~3.18(b)]{BKK}, the continuation map
between two of the chosen Hamiltonians $H$ and $H'$
sends the generator $p_k$ for $H$ to the generator $p_k$
for $H'$.
These identifications commute with the inclusion and quotient
maps associated with the action filtration.
For a fixed Hamiltonian, continuation between the almost
complex structures corresponding to different neck lengths
induces isomorphisms on filtered homology.

{For each $j$, \eqref{eq:mbdegree} gives only finitely many
generators in degrees $j-1,j,j+1$.
Once the action window contains all these generators,
further enlargement does not change the degree-$j$ homology.
The continuation maps send $p_k$ to $p_k$, so the limit
over Hamiltonians also stabilizes in degree $j$.}
Passing to these limits gives \eqref{eq:rfhqh} for
$\alpha_0$, as in \cite[Section~3.4.2]{BKK}.

For $\alpha=f\alpha_0$, realize $(Y,\alpha)$ as a graph in
the completion of $W$, after rescaling the Liouville form
if necessary.
The cylindrical completion of $C$ meets this graph in the
corresponding copy of $\Lambda$.
Continuation in the completion of $W$ gives
\eqref{eq:rfhqh} for $\alpha$, as in
\cite[Section~5.1]{BKK}.
\end{proof}

\subsection{Morse inequalities}
\label{subsec:morse-window-v2}

We work in the setting of Section~\ref{subsec:rfh-comparison}.
For $j\in\Z$, define the $j$th Betti number of RFH by
\begin{equation}\label{eq:rfh-betti-numbers}
 b_j^{\RFH}:=\dim_\F\RFH_j(\Lambda;\F).
\end{equation}

Suppose that every nonconstant relatively contractible
generalized chord of $\Lambda$ is isolated.
Define the \emph{Morse-type numbers} by
\begin{equation}\label{eq:rfh-morse-numbers}
 M_j:=\sum_x\dim_\F\mathrm{HF}_j^{\mathrm{loc}}(x;\F)
       +\dim_\F\mathrm{HF}_j^{\mathrm{loc}}(\Lambda;\F),
 \qquad j\in\Z,
\end{equation}
where the sum runs over all such generalized chords,
including both positive and negative chords.
The contribution of the family of constant chords is $b_j(L)$
by
\eqref{eq:local-floer-constant}.
The chords in this sum may be degenerate, and $M_j$ may
be infinite.
This definition is analogous to the one using local
equivariant symplectic homology in
\cite[Section~3.5]{GGM}.

If all these chords are nondegenerate,
\eqref{eq:local-floer-nondegenerate} gives
\begin{equation}\label{eq:rfh-morse-counts}
 M_j=\#\{x:\mu(x)=j\}+b_j(L),
\end{equation}
where $x$ ranges over the same generalized chords as in
\eqref{eq:rfh-morse-numbers}.

\begin{proposition}[Morse inequalities]
\label{prop:rfh-morse-inequalities}
Suppose that $(Y,\alpha,\rho)$ is a real contact manifold and
that $\Lambda$ is a union of connected components of $\Fix(\rho)$.
Assume that only finitely many geometrically distinct
$\rho$-brake Reeb orbits carry a positive relatively contractible
chord of $\Lambda$, and that their prime representatives
are strongly nondegenerate.
Let $x_1,\ldots,x_s$ be the shortest positive relatively
contractible chords of $\Lambda$, one for each initial
point admitting such a chord.
If $\widehat\mu(x_i)\ne0$ for $1\le i\le s$, then
\begin{equation}\label{eq:local-morse-inequalities}
 b_j^{\RFH}\le M_j<\infty
 \qquad(j\in\Z).
\end{equation}
\end{proposition}

\begin{proof}
By reflection and Lemma~\ref{lem:brake-admissible-iterates},
the nonconstant relatively contractible generalized chords are
$x_i^m$ and $(x_i^m)^{-1}$, with $1\le i\le s$ and $m\ge1$.
There are finitely many such families, since there are
finitely many prime orbits in the statement and each meets
$\Lambda$ in at most two points.
Strong nondegeneracy and
Lemma~\ref{lem:brake-capped-iteration} imply that all positive
chords are nondegenerate, and so are their reversals.
Since
\[
 \lim_{m\to\infty}\frac{|\mu(x_i^m)|}{m}
 =|\widehat\mu(x_i)|>0,
\]
only finitely many positive chords have degree in any fixed
bounded interval.
The same holds for negative chords by \eqref{eq:reversal}.
Thus \eqref{eq:rfh-morse-counts} gives $M_j<\infty$ for every $j$.

Fix integers $a\le b$.
Choose $A<0<B$ with
$A,B\notin\operatorname{Spec}(\alpha,\Lambda)$
so that every nonconstant relatively contractible generalized
chord in degrees $a-1,\ldots,b+1$ has $T\in(A,B)$.
Enlarging this action interval adds no generators in these
degrees.
The long exact sequences associated with the action
filtration therefore show that both limits in \eqref{eq:limits}
stabilize in degrees $a,\ldots,b$.
Choose a finite-dimensional Floer complex
$\mathrm{CF}_*^{(A,B)}$ computing $\RFH_*^{(A,B)}$, as in
Section~\ref{subsec:rfh-comparison}.
Then
\[
 H_j(\mathrm{CF}_*^{(A,B)})\cong\RFH_j(\Lambda;\F)
 \qquad(a\le j\le b).
\]
Here the chosen Hamiltonian is suppressed from the notation.

Choose the Morse perturbation sufficiently small that an
action interval $(-\varepsilon,\varepsilon)\subset(A,B)$
contains exactly the generators arising from the constant chords.
{The quotient complex for $(-\varepsilon,\varepsilon)$ computes
$\mathrm{HF}_*^{\mathrm{loc}}(\Lambda;\F)$, whose degree-$j$
dimension is $b_j(L)$ by \eqref{eq:local-floer-constant}.}
The long exact sequences associated with the action
filtration bound the homology dimension of the full complex
by the sum of the homology dimensions of the quotient complexes
for the intervals $(A,-\varepsilon)$, $(-\varepsilon,\varepsilon)$,
and $(\varepsilon,B)$.
The first and third quotient complexes have only nonconstant
generators, and their homology dimensions are bounded by
their chain dimensions.
Consequently,
\[
 b_j^{\RFH}
 \le \#\{x:\mu(x)=j\}+b_j(L)=M_j
 \qquad(a\le j\le b),
\]
where $x$ ranges over all nonconstant relatively contractible
generalized chords.
Since $a$ and $b$ were arbitrary, the result follows.
\end{proof}

For the proof of Theorem~\ref{thm:main-brake-critical}, we also
need a finite-dimensional chain complex whose degree-$j$ chain
group has dimension $M_j$ throughout a prescribed interval
of degrees.
Under the hypotheses of
Proposition~\ref{prop:rfh-morse-inequalities}, fix integers
$a\le b$ and choose $\mathrm{CF}_*^{(A,B)}$ as in the
preceding proof.
Let $\mathrm{CF}_*^0$ be the quotient complex associated with
the interval $(-\varepsilon,\varepsilon)$ chosen in that proof.
Over $\F$, this finite complex splits as a direct sum of
its homology with zero differential and two-term acyclic
complexes.
Make a homogeneous change of basis among the generators of
$\mathrm{CF}_*^{(A,B)}$ arising from the constant chords
to realize this decomposition, and cancel the corresponding
pairs in the full complex.
The change of basis preserves the two action subcomplexes
defining $\mathrm{CF}_*^0$, and the cancellations induce
the same reduction on $\mathrm{CF}_*^0$.

The resulting finite complex $\widetilde C_*$ is chain
homotopy equivalent to $\mathrm{CF}_*^{(A,B)}$.
{The cancellations preserve the number and degrees of the
nonconstant generators and leave exactly $b_j(L)$ generators
arising from the constant chords in degree $j$.}
The latter correspond to a homogeneous basis of
$H_*(\mathrm{CF}_*^0)$, and the component of the differential
between them is zero.
By the choice of $(A,B)$,
\begin{equation}\label{eq:window-full-identification-v2}
 H_j(\widetilde C_*)\cong\RFH_j(\Lambda;\F),\qquad
 \dim_\F\widetilde C_j=M_j
 \qquad(a\le j\le b).
\end{equation}

\subsection{Index estimates}
\label{subsec:mean-index-estimates}
\label{subsec:brake-iteration}

\begin{lemma}\label{lem:rfh-mean-density}
{Assume the hypotheses and notation of
Proposition~\ref{prop:rfh-morse-inequalities}.}
If there exist an integer $J_0$ and a number $\eta\ge0$ such that
\[
 \sum_{j=J_0}^{R}\dim_{\F}\RFH_j(\Lambda;\F)
 =\eta R+O(1)
 \qquad(R\in\Z,\ R\to\infty),
\]
then
\begin{equation}\label{eq:rfh-mean-density}
 \eta\le\sum_{i=1}^{s}\frac1{|\widehat\mu(x_i)|}.
\end{equation}
\end{lemma}

\begin{proof}
{By Lemma~\ref{lem:brake-admissible-iterates}, every
nonconstant relatively contractible generalized chord
is uniquely of the form $x_i^m$ or $(x_i^m)^{-1}$,
with $1\le i\le s$ and $m\ge1$.}
For each $i$, consider $x_i^m$ if $\widehat\mu(x_i)>0$ and
$(x_i^m)^{-1}$ if $\widehat\mu(x_i)<0$.
By \eqref{eq:reversal}, their degrees are
\[
 q_i(m):=
 \begin{cases}
  \mu(x_i^m),&\widehat\mu(x_i)>0,\\
  n-\mu(x_i^m),&\widehat\mu(x_i)<0.
 \end{cases}
\]
By \eqref{eq:brake-mean-definition} and
\eqref{eq:degree},
\[
 \frac{q_i(m)}m\longrightarrow|\widehat\mu(x_i)|>0.
\]
Their reversals have degrees $n-q_i(m)$, which tend to
$-\infty$ in each family.
Since there are finitely many families, choose an integer
$J\ge J_0$ greater than $n$ and all the degrees $n-q_i(m)$.

The family of constant chords contributes only in degrees
$0,\ldots,n$.
Thus, in degrees at least $J$, only the chosen chords
contribute to the Morse-type numbers.
Fix $0<\varepsilon<1$.
{Since $q_i(m)/m\to|\widehat\mu(x_i)|$, for each $i$ we have}
\[
 q_i(m)\ge(1-\varepsilon)|\widehat\mu(x_i)|m
\]
for all sufficiently large $m$.
Hence
\[
 \#\{m\ge1:J\le q_i(m)\le R\}
 \le\frac{R}{(1-\varepsilon)|\widehat\mu(x_i)|}+O(1).
\]
Proposition~\ref{prop:rfh-morse-inequalities} and
\eqref{eq:rfh-morse-counts} give
\[
 \begin{aligned}
 \eta R+O(1)
 &=\sum_{j=J}^{R}\dim_\F\RFH_j(\Lambda;\F)\\
 &\le\sum_{i=1}^{s}\#\{m\ge1:J\le q_i(m)\le R\}\\
 &{\le\frac{R}{1-\varepsilon}
       \sum_{i=1}^{s}\frac1{|\widehat\mu(x_i)|}+O(1).}
 \end{aligned}
\]
{Dividing by $R$ and letting $R\to\infty$, then
$\varepsilon\to0$, proves \eqref{eq:rfh-mean-density}.}
\end{proof}

\begin{corollary}\label{lem:brake-mean-density}
In the setting of Conjecture~\ref{conj:real-prequantization-bound},
suppose that $L$ is connected, positively monotone, and wide
over $\F$, with finite minimal Maslov number $N_L\ge2$.
Assume that $f\circ\delta=f$ and, if $N_L=2$, that
condition~\eqref{ass:fillings} holds.
Suppose that there are finitely many geometrically distinct
$\rho$-brake Reeb orbits and that every prime $\rho$-brake Reeb
orbit carrying a positive relatively contractible chord of
$\Lambda$ is strongly nondegenerate.
On the quotient $\bar Y=Y/\langle\delta\rangle$, let $x_1,\ldots,x_s$ be the
shortest positive relatively contractible chords of
$\bar\Lambda$, one for each initial point admitting such a chord.
If $\widehat\mu(x_i)\ne0$ for $1\le i\le s$, then
\begin{equation}\label{eq:brake-mean-density}
 \frac{b(L)}{N_L}\le\sum_{i=1}^s\frac1{|\widehat\mu(x_i)|}.
\end{equation}
Moreover, the Reeb flow of $\alpha$ has at least $s$
geometrically distinct $\rho$-brake Reeb orbits.
\end{corollary}

\begin{proof}
\leavevmode By Section~\ref{subsec:brake-quotient}, only finitely many prime
Reeb orbits on $\bar Y$ carry a positive relatively contractible
chord of $\bar\Lambda$, and all such orbits are strongly
nondegenerate.
Since $\bar\Lambda$ is a union of connected components of
$\Fix(\bar\rho)$ and $\widehat\mu(x_i)\ne0$,
Proposition~\ref{prop:rfh-morse-inequalities} applies on the quotient.
Lemma~\ref{lem:real-half-period-quotient} gives at least $s$
geometrically distinct $\rho$-brake Reeb orbits in $Y$.

By \eqref{eq:rfhqh}, established by
Lemma~\ref{lem:relative-filling-comparison} when $N_L=2$,
and \eqref{eq:wide-grading},
\[
 \sum_{j=J}^{R}\dim_{\F}\RFH_j(\bar\Lambda;\F)
 =\frac{b(L)}{N_L}R+O(1)
\]
for every fixed integer $J$ as $R\to\infty$.
Applying Lemma~\ref{lem:rfh-mean-density} to $\bar\Lambda$
with $\eta=b(L)/N_L$ proves \eqref{eq:brake-mean-density}.
\end{proof}

\begin{lemma}\label{lem:brake-single-iterate-window}
In the setting of Lemma~\ref{lem:brake-capped-iteration},
let $x_1,\ldots,x_s$ be positive relatively contractible
chords.
Suppose that, for each $i$, every iterate of the
associated $\rho$-brake Reeb orbit $x_i^2$ is nondegenerate.
Assume that, for each $i$,
\begin{equation}\label{eq:brake-window-index-assumption}
 \mu_I(x_i),\mu_{-I}(x_i)\ge\frac n2
 \quad\text{or}\quad
 \mu_I(x_i),\mu_{-I}(x_i)\le-\frac n2.
\end{equation}
Then the following hold.
\begin{enumerate}
\item[(i)]
We have $\widehat\mu(x_i)\ne0$ for every $i$.
Set
\begin{equation}\label{eq:brake-signed-degrees}
 \varepsilon_i=\operatorname{sign}\widehat\mu(x_i),\qquad
 q_i(\ell)=\varepsilon_i\mu_I(x_i^\ell)+\frac n2.
\end{equation}
Each sequence $q_i(\ell)$ is nondecreasing and satisfies
$q_i(\ell)\ge n$ for all $\ell\ge1$.

\item[(ii)]
For every integer $D\ge1$, there are arbitrarily large
$K\in D\mathbb N$ and positive integers $m_1,\ldots,m_s$
such that, for each $i$ and every $\ell\ge1$,
\[
 q_i(\ell)\in[K+1,K+n-1]
 \quad\Longrightarrow\quad
 \ell=2m_i.
\]
For each $i$ satisfying
\begin{equation}\label{eq:brake-window-strong-index-assumption}
 \varepsilon_i\mu_I(x_i),\varepsilon_i\mu_{-I}(x_i)
 \ge\frac n2+1,
\end{equation}
the same implication holds with $[K,K+n]$ in place of
$[K+1,K+n-1]$.
\end{enumerate}
\end{lemma}

\begin{proof}
By \eqref{eq:real-double-index} and
\eqref{eq:mean-cz-estimate},
\[
 \left|2\widehat\mu(x_i)-\mu_I(x_i)-\mu_{-I}(x_i)\right|<n.
\]
The two alternatives in
\eqref{eq:brake-window-index-assumption} therefore imply
$\widehat\mu(x_i)>0$ and $\widehat\mu(x_i)<0$, respectively.

For each $i$, choose a relative capping of $x_i$ and an
adapted symplectic trivialization as in
Lemma~\ref{lem:brake-capped-iteration}.
Let $\Phi_i$ be the symplectic path representing the
linearized Reeb flow along $x_i$ in this trivialization.
By \eqref{eq:brake-signed-index-growth},
\[
 q_i(\ell+1)-q_i(\ell)
 \ge\varepsilon_i\mu_I(x_i)-\frac n2\ge0.
\]
Since $q_i(1)\ge n$, this proves part~(i).

Fix an integer $D\ge1$.
Theorem~\ref{thm:brake-common-jump-standard}, applied to
$\Phi_1,\ldots,\Phi_s$, gives arbitrarily large
$K\in D\mathbb N$ and positive integers $m_1,\ldots,m_s$.
Combining \eqref{eq:brake-common-jump-standard} with
\eqref{eq:brake-capped-iteration} and
\eqref{eq:brake-signed-degrees}, we obtain
\begin{equation}\label{eq:brake-window-endpoints}
 \begin{aligned}
  q_i(2m_i-1)
   &=K-\varepsilon_i\mu_{-I}(x_i)+\frac n2\le K,\\
  q_i(2m_i+1)
   &=K+\varepsilon_i\mu_I(x_i)+\frac n2\ge K+n.
 \end{aligned}
\end{equation}
Since $q_i(\ell)$ is nondecreasing, these inequalities imply
that $q_i(\ell)\in[K+1,K+n-1]$ is possible only when
$\ell=2m_i$.

For each $i$ satisfying
\eqref{eq:brake-window-strong-index-assumption},
\eqref{eq:brake-window-endpoints} gives
\[
 q_i(2m_i-1)\le K-1,\qquad
 q_i(2m_i+1)\ge K+n+1.
\]
Again by monotonicity, $q_i(\ell)\in[K,K+n]$ is possible
only when $\ell=2m_i$.
This proves part~(ii).
\end{proof}

By \eqref{eq:degree} and \eqref{eq:reversal}, the number
$q_i(\ell)$ is the Floer degree of $x_i^\ell$ when
$\varepsilon_i=1$, and of $(x_i^\ell)^{-1}$ when
$\varepsilon_i=-1$.
In either case, the reversal has Floer degree
$n-q_i(\ell)\le0$.

\section{Proofs of the main results}
\label{sec:brake-total-proof}

\subsection{Proofs of Theorems~\ref*{thm:main-brake-total-betti} and~\ref*{thm:main-brake-critical}}
\label{subsec:brake-total-proof}

Under the hypotheses of either
Theorem~\ref{thm:main-brake-total-betti} or
Theorem~\ref{thm:main-brake-critical}, we have $N_L\ge n\ge2$,
and condition~\eqref{ass:fillings} is assumed when $N_L=2$.
We may assume that there are only finitely many geometrically
distinct $\rho$-brake Reeb orbits.

We pass to the quotient $\bar Y=Y/\langle\delta\rangle$,
equipped with the contact form $\bar\alpha$ and the Legendrian
$\bar\Lambda$ from Lemma~\ref{lem:real-half-period-quotient}.
By Section~\ref{subsec:brake-quotient}, \eqref{eq:rfhqh}
holds for $\bar\Lambda$.
When $N_L=2$, this follows from
Lemma~\ref{lem:relative-filling-comparison} under
\eqref{ass:fillings}.
By Lemma~\ref{lem:real-two-boundary-indices}
and \eqref{eq:degree}, the covering map preserves the
Maslov indices and the Floer grading.
Moreover, for every positive relatively contractible
chord $x$ of $\bar\Lambda$, the associated $\bar\rho$-brake
Reeb orbit $x^2$ and all its iterates are nondegenerate.

Let $M_j$ denote the Morse-type numbers defined by
\eqref{eq:rfh-morse-numbers} for the contact form $\bar\alpha$
and the Legendrian $\bar\Lambda$.
Let $\Gamma$ be the set of prime Reeb orbits in $\bar Y$ that
{carry a positive relatively contractible chord of
$\bar\Lambda$}. Each orbit in $\Gamma$ is covered by a prime
$\rho$-brake Reeb orbit in $Y$. The finiteness assumption
therefore implies that $\Gamma$ is finite.

For each $\bar\gamma\in\Gamma$, let
$s_{\bar\gamma}=\#(\bar\gamma\cap\bar\Lambda)$.
Then $s_{\bar\gamma}\in\{1,2\}$.
For each $p\in\bar\gamma\cap\bar\Lambda$, let $x_p$ be the
shortest positive relatively contractible chord of
$\bar\Lambda$ starting at $p$.
By Lemma~\ref{lem:brake-admissible-iterates}, every positive
relatively contractible chord of $\bar\Lambda$ is an
iterate of one of the chords $x_p$.
Each $x_p$ lifts to a positive relatively contractible
chord of $\Lambda$.
By \eqref{eq:main-brake-total-index} or
\eqref{eq:main-brake-critical-index}, every $x_p$ therefore
satisfies \eqref{eq:brake-window-index-assumption}.

Put $s=\sum_{\bar\gamma\in\Gamma}s_{\bar\gamma}$.
By Lemma~\ref{lem:real-half-period-quotient}, exactly $s$
geometrically distinct $\rho$-brake Reeb orbits in $Y$ project
onto the orbits in $\Gamma$.
If $s=0$, then $M_j=b_j(L)$ for every $j$.
Applying Proposition~\ref{prop:rfh-morse-inequalities}
to $(\bar Y,\bar\alpha,\bar\Lambda)$ would then give
$\dim_\F\RFH_j(\bar\Lambda;\F)=0$ for every $j>n$.
This contradicts \eqref{eq:rfhqh} and \eqref{eq:wide-grading},
which give nonzero RFH in arbitrarily large degrees.
Thus $s\ge1$.

By Lemma~\ref{lem:brake-single-iterate-window}, each $x_p$ has
nonzero mean index. Put
\[
 \varepsilon_p=\operatorname{sign}\widehat\mu(x_p),\qquad
 q_p(m)=\varepsilon_p\mu_I(x_p^m)+\frac n2
 \qquad(m\ge1).
\]
Each sequence $q_p(m)$ is nondecreasing and satisfies
$q_p(m)\ge n$ for all $m\ge1$.
For each $p$ and $m\ge1$, select $x_p^m$ if $\varepsilon_p=1$
and $(x_p^m)^{-1}$ if $\varepsilon_p=-1$.
By \eqref{eq:reversal}, the selected chord has degree $q_p(m)$,
whereas its reversal has degree $n-q_p(m)\le0$.

Proposition~\ref{prop:rfh-morse-inequalities} applies on the
quotient.
Only the family of constant chords contributes to the
Morse-type numbers in degrees $1,\ldots,n-1$;
in degrees greater than $n$,
only the selected chords contribute.

\begin{proof}[Proof of Theorem~\ref{thm:main-brake-total-betti}]
\leavevmode By \eqref{eq:main-brake-total-index} and
Lemma~\ref{lem:brake-single-iterate-window}, choose
$K\in N_L\mathbb N$, $K>n$, so that, for each $p$, at most
one integer $m\ge1$ satisfies $q_p(m)\in[K,K+n]$.
Since $N_L>n$ and $N_L\mid K$, \eqref{eq:rfhqh},
Proposition~\ref{prop:rfh-morse-inequalities}, and
\eqref{eq:wide-grading} give
\begin{equation}\label{eq:brake-window-betti-bound}
 b(L)=\sum_{j=0}^{n}\dim_{\F}\QH_{K+j}(L;\F)
 \le\sum_{j=0}^{n}M_{K+j}
 \le s.
\end{equation}
\end{proof}

\begin{proof}[Proof of Theorem~\ref{thm:main-brake-critical}]
\leavevmode Here $N_L=n$.
When $b(L)=2$, the first conclusion follows from $s\ge1$.
To prove this conclusion, we may therefore assume that
$b(L)>2$.

Apply Lemma~\ref{lem:brake-single-iterate-window} with
$D=2N_L$ to obtain $K\in2N_L\mathbb N$, $K>n$, and
positive integers $m_p$ such that, for each $p$,
$q_p(m)\in[K+1,K+n-1]$ is possible only when $m=2m_p$.
Put $J=K+n-1$.
Since $N_L=n$ and $N_L\mid K$, \eqref{eq:wide-grading} gives
\begin{equation}\label{eq:brake-interior-quantum-ranks}
 \dim_{\F}\QH_{K+j}(L;\F)=b_j(L)
 \qquad(1\le j\le n-1).
\end{equation}
Since $b_0(L)=b_n(L)=1$, \eqref{eq:rfhqh}
and Proposition~\ref{prop:rfh-morse-inequalities} yield
\begin{equation}\label{eq:brake-window-upper-bound}
 b(L)-2
 =\sum_{j=1}^{n-1}b_j(L)
 \le\sum_{j=K+1}^{J}M_j
 \le s.
\end{equation}

Suppose, for a contradiction, that $s=b(L)-2$.
Both inequalities in \eqref{eq:brake-window-upper-bound}
are then equalities.
Thus $q_p(2m_p)\in[K+1,J]$ for every $p$, and
\begin{equation}\label{eq:brake-interior-equality}
 M_j=\dim_{\F}\QH_j(L;\F)
 \qquad(K+1\le j\le J).
\end{equation}

We compare the sums of the dimensions of the chain groups
and the homology groups in degrees $n,\ldots,J$ modulo two.
Choose the finite complex $\widetilde C_*$ constructed at
the end of Section~\ref{subsec:morse-window-v2} for the
degree interval $[0,J+1]$.
By \eqref{eq:window-full-identification-v2} and
\eqref{eq:rfhqh}, we have
\[
 H_j(\widetilde C_*)\cong\QH_j(L;\F),\qquad
 \dim_\F\widetilde C_j=M_j
 \qquad(0\le j\le J+1).
\]
Put
$r_j=\rank(\partial:\widetilde C_j\to\widetilde C_{j-1})$.
The rank--nullity theorem gives
\begin{equation}\label{eq:morse-rank-v2}
 \dim_\F\widetilde C_j
 =\dim_\F\QH_j(L;\F)+r_j+r_{j+1}
 \qquad(0\le j\le J+1).
\end{equation}
{Since $q_p(m)\ge n$, $n-q_p(m)\le0$, and $N_L=n$, we have}
\[
 \dim_\F\widetilde C_j=b_j(L)=\dim_\F\QH_j(L;\F)
 \qquad(1\le j\le n-1).
\]
Applying \eqref{eq:morse-rank-v2} in degree $n-1$ gives
$r_n=0$.
Applying it in degree $J$ and using
\eqref{eq:brake-interior-equality} gives $r_{J+1}=0$.

{Since $q_p(2m_p)\le J$,
Lemma~\ref{lem:brake-single-iterate-window}(i) and
\eqref{eq:brake-window-endpoints} show that, for each $p$,
the selected chords of degree in $[n,J]$ are precisely
those with $1\le m\le2m_p$.
Since $b_n(L)=1$, exactly one generator arising from the
constant chords has degree in $[n,J]$.}
Thus
\begin{equation}\label{eq:brake-prefix-generator-count}
 \sum_{j=n}^{J}\dim_\F\widetilde C_j
 =1+2\sum_p m_p.
\end{equation}
Since $J=K+n-1$ and $N_L\mid K$,
\eqref{eq:wide-grading} gives
\begin{equation}\label{eq:brake-prefix-quantum-count}
 \sum_{j=n}^{J}\dim_\F\QH_j(L;\F)
 =\frac{K}{N_L}b(L).
\end{equation}
Summing \eqref{eq:morse-rank-v2} over $j=n,\ldots,J$
and using $r_n=r_{J+1}=0$ gives
\begin{equation}\label{eq:brake-prefix-parity}
 1+2\sum_p m_p
 =\frac{K}{N_L}b(L)+2\sum_{j=n+1}^{J}r_j.
\end{equation}
The left side is odd, whereas the right side is even
because $K\in2N_L\mathbb N$.
Thus $s\ge b(L)-1$, which proves the first conclusion.

For the additional conclusion, assume
\eqref{eq:main-brake-critical-mean}.
Since the covering preserves the Maslov indices of all
iterates, it also preserves the mean index.
Hence
\[
 |\widehat\mu(x_p)|>\frac{n(b(L)-1)}{b(L)}.
\]
 By \eqref{eq:brake-mean-density},
\[
 \frac{b(L)}n
 \le\sum_p\frac1{|\widehat\mu(x_p)|}
 <\frac{s\,b(L)}{n(b(L)-1)}.
\]
Since $b(L)\ge2$, this gives $s>b(L)-1$, and hence
$s\ge b(L)$.
\end{proof}

\subsection{Contact geometry of Finsler spheres}
\label{subsec:sphere-contact-model}

Let $F$ be a Finsler metric on $S^{n+1}$, $n\ge1$, and let
$\sigma:S^{n+1}\to S^{n+1}$ be the restriction of a nontrivial
orthogonal linear involution of $\R^{n+2}$.
Assume that $\Fix(\sigma)\ne\varnothing$ and that $F$ is
real with respect to $\sigma$.
The fixed point set of $\sigma$ is then a great sphere,
which we denote by $S^d$, with $0\le d\le n$.
Set $a=\min\{d,n-d\}$, as in
Section~\ref{subsec:conjectures}.

The Legendre transform
$\mathcal L_F(q,v)=(q,d_v(F^2/2))$
identifies the unit tangent bundle
$S_FS^{n+1}$
with the unit cosphere bundle $Y_F$.
It intertwines the involution
$(q,v)\mapsto(\sigma(q),-(D\sigma)_q(v))$ with $\rho$.
By \eqref{eq:sphere-finsler-contact}, the cotangent bundle
projection maps a Reeb trajectory of $\alpha_F$ on $[0,T]$
to an $F$-geodesic segment of length $2\pi T$.

Consider the unit cosphere bundle $Y_0=S_{g_0}^*S^{n+1}$
with contact form $\alpha_0=\theta|_{Y_0}/(2\pi)$.
The radial map
\[
 r_F:Y_0\longrightarrow Y_F,\qquad
 r_F(q,p)=\left(q,\frac{p}{F^*(q,p)}\right)
\]
intertwines the involutions and satisfies
\[
 r_F^*\alpha_F=f\alpha_0,\qquad
 f=(F^*|_{Y_0})^{-1}>0,\qquad
 f\circ\rho=f.
\]
Using $g_0$, identify each unit covector $p$ at $q$
with the unit tangent vector $v$ determined by
$p=g_0(v,\cdot)$.
The Reeb flow of $\alpha_0$ rotates the orthonormal pair
$(q,v)$ with angular speed $2\pi$.
The map $(q,v)\mapsto[v+iq]$ identifies the
orbit space of this flow with the complex quadric
\[
 Q^n=\left\{[z_0:\cdots:z_{n+1}]\in\mathbb{CP}^{n+1}:
              \sum_{j=0}^{n+1}z_j^2=0\right\}.
\]
The time-$1/2$ map $\delta=\phi_{\alpha_0}^{1/2}$ is given by
$\delta(q,v)=(-q,-v)$.

For a submanifold $N$ of a Finsler manifold $(M,F)$,
denote its \emph{conormal and unit conormal bundles} by
\[
 N^*N=\{(q,p)\in T^*M:q\in N,\ p|_{T_qN}=0\},
 \qquad S^*_{F,N}M=N^*N\cap S_F^*M.
\]
The radial map $r_F^{-1}:Y_F\to Y_0$ identifies the fixed
Legendrian $\Lambda=S^*_{F,S^d}S^{n+1}$ with
$S^*_{g_0,S^d}S^{n+1}\subset Y_0$.
At each $q\in S^d$, the unit normal vectors to $S^d$
with respect to $g_0$ form the unit sphere in the
$-1$ eigenspace of the underlying linear involution.
The fiber of $Y_0\to Q^n$ through such a point $(q,v)$
meets the unit conormal bundle precisely at $(q,v)$
and $(-q,-v)$.
The prequantization projection $Y_0\to Q^n$
maps this unit conormal bundle onto $L$, and
\begin{equation}\label{eq:sphere-pinching-real-locus}
 \Lambda=\Fix(\rho)\cong S^d\times S^{n-d},\qquad
 L\cong(S^d\times S^{n-d})/\bigl((x,y)\sim(-x,-y)\bigr)
 \subset Q^n.
\end{equation}
For $n\ge2$, this description of the real locus also appears
in \cite{IMMO}.

For $a\ge1$, the cohomology calculation in
\cite[Theorem~2.1]{DavisPPS} gives
\[
 H^j(L;\F)\cong
 H^j(\mathbb{RP}^a;\F)\oplus
 H^{j-(n-a)}(\mathbb{RP}^a;\F)
\]
as $\F$-vector spaces for every $j\in\Z$.
Taking dimensions and summing over $j$, we obtain
\[
 b(L)=2\sum_{j=0}^{a}\dim_\F H^j(\mathbb{RP}^a;\F)
     =2(a+1).
\]
When $a=0$, we have $L\cong S^n$, so the same formula holds.

Thus Conjecture~\ref{conj:real-prequantization-bound}
implies Conjecture~\ref{conj:sphere-brake-bound}.

For the remainder of this subsection, assume $n\ge2$.
The Lagrangian $L\subset Q^n$ is positively monotone and wide.
Its minimal Maslov number and total Betti number are
\begin{equation}\label{eq:sphere-pinching-maslov}
 N_L=
 \begin{cases}
  2n,&a=0,\\
  n,&a\ge1,
 \end{cases}
 \qquad b(L)=2a+2.
\end{equation}
Positive monotonicity and the formula for $N_L$ follow from
\cite[Proposition~2.1 and Remark~1]{IMMO}.

{For $N_L\ge3$, wideness follows from
\cite[Theorem~B]{OhRealForms}.}
In the remaining case $n=2$, $a=1$, we identify $L\subset Q^2$ with
the product of equators
$\mathbb{RP}^1\times\mathbb{RP}^1\subset
\mathbb{CP}^1\times\mathbb{CP}^1$.
For the product complex structure, the four holomorphic disks of
Maslov index two through a fixed boundary point are the hemispheres
in each factor, with the other factor constant.
Their boundary classes cancel in $H_1(L;\F)$, so the differential of
the minimal pearl complex vanishes on $H_1(L;\F)$ by
\cite[formula~(22)]{BC}.
Since $H_*(L;\F)$ is generated by $H_1(L;\F)$ and the fundamental
class under the intersection product, the argument in
\cite[proof of Lemma~5.1.1]{BC} gives the vanishing of the differential.
Thus $L$ is wide.

Suppose first that $a\ge1$, and assume
\eqref{eq:sphere-antipodal-symmetry}.
Then the cotangent lift of the antipodal map,
$\delta(q,p)=(-q,-p)$, preserves $Y_F$ and satisfies
$\delta^*\alpha_F=\alpha_F$.
The radial map $r_F$ commutes with $\delta$, and
$f\circ\delta=f$.
The metric $F$ descends to $\mathbb{RP}^{n+1}$, and we use
the same notation for the descended metric.
The quotient by $\delta$ is then
\begin{equation}\label{eq:sphere-pinching-conormal-pair}
 (\bar Y,\bar\Lambda)
 =\bigl(S_F^*\mathbb{RP}^{n+1},
         S^*_{F,\mathbb{RP}^d}\mathbb{RP}^{n+1}\bigr).
\end{equation}
Equation~\eqref{eq:sphere-pinching-real-locus}
shows that $\Lambda$ is connected.
Since $Y_F$ is diffeomorphic to the simply connected manifold
$S_{g_0}^*S^{n+1}$, every chord of $\Lambda$ is relatively
contractible.
A path in $\Lambda$ from a point $p$ to $\delta(p)$
projects to a loop in $\bar\Lambda$ representing the nontrivial
element of $\pi_1(\bar Y)\cong\mathbb Z_2$.
Hence the inclusion-induced map
$\pi_1(\bar\Lambda)\to\pi_1(\bar Y)$ is surjective,
and every chord of $\bar\Lambda$ is relatively
contractible as well.

For $n\ge3$, \eqref{eq:rfhqh} holds for $\bar\Lambda$.
If $n=2$, then $a=d=1$, so $L\cong T^2$ and $N_L=2$.
Let $\varpi:S^3\to Q:=\mathbb{RP}^3$ be the antipodal quotient
map, and set $N:=\varpi(\operatorname{Fix}(\sigma))\subset Q$,
so that $N\cong\mathbb{RP}^1$.
Define
\begin{equation}\label{eq:sphere-three-filling}
 \begin{aligned}
 D_F^*Q&=\{(q,p)\in T^*Q:F^*(q,p)\le1\},\\
 (W_F,C_F)&=\bigl(D_F^*Q,\,N^*N\cap D_F^*Q\bigr).
 \end{aligned}
\end{equation}

\begin{lemma}\label{lem:sphere-three-filling}
The manifold $W_F$ is a Liouville filling of $\bar Y$,
and $C_F$ is an exact Lagrangian filling of $\bar\Lambda$.
The pair $(W_F,C_F)$ satisfies condition~\eqref{ass:fillings}.
\end{lemma}

\begin{proof}
The Liouville form $\theta/(2\pi)$ vanishes on $N^*N$.
Thus $C_F$ is an exact Lagrangian filling of $\bar\Lambda$,
and its cylindrical completion is $N^*N$.

The projection $\bar Y=S_F^*Q\to Q$ has simply connected
fiber $S^2$, and $W_F$ deformation retracts onto $Q$.
The projection $\bar\Lambda\to N$ has connected fiber $S^1$,
and $C_F$ deformation retracts onto $N$.
Since $\pi_1(N)\to\pi_1(Q)$ is surjective, the homotopy
exact sequences of the two pairs give
\[
 \pi_1(\bar Y,\bar\Lambda)=\pi_1(W_F,C_F)=0.
\]

The Maslov class of $C_F$ also vanishes on $\pi_2(W_F,C_F)$.
To see this, take a disk in $W_F$ with boundary on $C_F$.
Scaling the cotangent fibers deforms this disk into the zero
section while keeping its boundary on $C_F$.
Choose a metric splitting and an orthonormal trivialization
of $TQ$ over the projected disk.
In the induced trivialization, the tangent spaces to $N^*N$
along the boundary of the deformed disk have the form
$V(t)\oplus iV(t)^\perp$.
Here $q(t)$ is the projected boundary loop, and $V(t)$
is the image of $T_{q(t)}N$ under the chosen trivialization.
Complex conjugation preserves each of these Lagrangian
subspaces and changes the sign of the Maslov index of the loop.
This index is therefore zero; see also
\cite[Proposition~1.2]{APS}.
\end{proof}

Applying Lemma~\ref{lem:relative-filling-comparison} to
$(W_F,C_F)$ gives \eqref{eq:rfhqh} also when $n=2$ and $a=1$.

When $a=0$, we have $d=0$ or $d=n$, and the fixed
Legendrian has two connected components:
\[
 \Lambda=\Lambda_+\sqcup\Lambda_-,\qquad
 \Lambda_+\cong\Lambda_-\cong S^n.
\]
Since $S_{g_0}^*S^{n+1}$ is simply connected, every chord
of either component is relatively contractible.
Each of $r_F^{-1}(\Lambda_+)$ and $r_F^{-1}(\Lambda_-)$ in
$Y_0$ is a Legendrian lift that projects diffeomorphically
onto $L\cong S^n$.
Applying \cite[Theorem~1.1]{BKK} to each component and using
\eqref{eq:wide-grading} with $N_L=2n$, we obtain
\begin{equation}\label{eq:sphere-component-rfh}
 \RFH_j(\Lambda_\pm;\F)
 \cong\QH_j(L;\F)
 \cong\bigoplus_{k\in\Z}H_{j-2nk}(S^n;\F).
\end{equation}

\subsection{Proof of Theorem~\ref*{cor:sphere-pinching}}
\label{subsec:sphere-index-estimates}

\leavevmode We assume the hypotheses of Theorem~\ref{cor:sphere-pinching},
including \eqref{eq:sphere-standard-pinching}, and use the notation
of Section~\ref{subsec:sphere-contact-model}.

\begin{lemma}\label{lem:sphere-pinching-index-estimates}
Every chord $x:[0,T]\to Y_F$ in $\cC^0_+(\alpha_F,\Lambda)$ satisfies
\begin{equation}\label{eq:sphere-boundary-index-bound}
 \mu_I(x),\ \mu_{-I}(x)\ge\frac n2.
\end{equation}
If $x(0)=x(T)$, then
\begin{equation}\label{eq:sphere-pinching-chord-mean}
 \widehat\mu(x)>n.
\end{equation}
If \eqref{eq:sphere-pinching} also holds, then
\begin{equation}\label{eq:sphere-pinching-mean-threshold}
 \widehat\mu(x)>\frac{n(2a+1)}{2a+2}.
\end{equation}
\end{lemma}

\begin{proof}
Let $x:[0,T]\to Y_F$ be a chord in
$\cC^0_+(\alpha_F,\Lambda)$, and let $c_x$ be its projection
to $S^{n+1}$.
Write $\ell(x)$ for the Finsler length of $c_x$.
We denote by $i_d(x)$ the Morse index of $c_x$ for the
energy functional, with both endpoints free to move on $S^d$.

After a constant rescaling of time, $x$ is a trajectory
of the energy Hamiltonian $H=(F^*)^2/2$.
Since $H$ is smooth near $x$ and its fiberwise Hessian
is positive definite along $x$, the Morse index theorem
for conormal boundary conditions applies; see
\cite[Theorem~4.1 and Remark~4.3]{APS}.

Along $x$, the tangent bundle $T(T^*S^{n+1})$ splits as
the symplectic direct sum of the contact distribution
and the plane spanned by the radial and Reeb directions.
The linearized Hamiltonian flow preserves both summands.
At each endpoint, the tangent space of the conormal
bundle $N^*S^d$ is the direct sum of the tangent space
of $\Lambda$ and the radial line.
On the contact distribution, the linearized Hamiltonian
flow agrees with the linearized Reeb flow after the
time rescaling above.
With Lagrangian boundary conditions
$T_{x(0)}\Lambda$ and $T_{x(T)}\Lambda$, this restriction
has Robbin--Salamon index $\mu_I(x)$ by
\eqref{eq:real-boundary-indices}.

To compute the contribution from the radial and Reeb directions,
let $Z$ be the fiberwise radial vector field and let $\phi_H^t$
denote the Hamiltonian flow of $H$ with respect to $d\theta$.
Since $H$ is homogeneous of degree two in the fibers,
\[
 (D\phi_H^t)_{x(0)}\bigl(Z_{x(0)}\bigr)
 =Z_{x(t)}+tX_H(x(t)).
\]
The conormal boundary conditions in this plane are given by
the radial direction.
{The contribution of this plane is the Robbin--Salamon
index of the pair
\[
 t\longmapsto
 \left(
 (D\phi_H^t)_{x(0)}\bigl(\mathbb R Z_{x(0)}\bigr),
 \mathbb R Z_{x(t)}
 \right).
\]}
The two lines coincide at $t=0$ and are transverse for
every $t>0$.
In the frame $(Z,X_H)$, with the convention
$\iota_{X_H}d\theta=-dH$, the crossing form at $t=0$ is
\[
 \Gamma_0(sZ_{x(0)})
 =s^2\,d\theta\bigl(Z_{x(0)},X_H(x(0))\bigr)
 =2H(x(0))s^2
 =s^2,
\]
where the last equality follows from $F^*(x(0))=1$.
Its signature is one, so this plane contributes $1/2$
to the Robbin--Salamon index; see \cite[(2)]{APS}.
{The adapted trivialization of the contact distribution and
the frame $(Z,X_H)$ together give a symplectic trivialization
over the relative capping.
Along the capping arc, the conormal boundary condition is
constant in this trivialization.
By additivity under symplectic direct sums, the
Robbin--Salamon index in this trivialization is
$\mu_I(x)+1/2$.}

{To apply the Morse index formula in
\cite[Theorem~4.1]{APS}, we show that its Robbin--Salamon
term equals $\mu_I(x)+1/2$.}
Choose a frame of the tangent bundle of $S^{n+1}$ over
the projected relative capping.
{Together with the horizontal splitting defined by $g_0$,
this frame gives a second symplectic trivialization over
the capping.
In this trivialization, the tangent spaces of the cotangent
fibers are identified with $\{0\}\times(\R^{n+1})^*$; see
\cite[Section~3]{APS}.}
Parametrize the capping arc from $x(0)$ to $x(T)$, and let
$\mathcal L(s)$, $0\le s\le1$, denote the images of the
tangent spaces of $N^*S^d$ along this arc in that trivialization.
Since $S^d$ is totally geodesic for $g_0$, each
$\mathcal L(s)$ is a conormal subspace of $T^*\R^{n+1}$.

{Since both trivializations extend over the capping,
the restrictions of their transition map to the chord
and to the capping arc are homotopic with fixed endpoints.}
By homotopy invariance, naturality, and concatenation,
{the Robbin--Salamon index in the second trivialization is
\[
 \mu_I(x)+\frac12
 +\mu_{\RS}\bigl(\mathcal L,\mathcal L(1)\bigr).
\]
The last term vanishes by \cite[Proposition~1.2]{APS}.}

By Lemma~\ref{lem:brake-capped-iteration}, the Reeb chord
$x$ is nondegenerate.
{The linearized Hamiltonian flow also sends the initial
radial line to a line transverse to the final radial line.}
The full linearized Hamiltonian equation therefore has
no nonzero solution satisfying the conormal boundary
conditions, so the nullity term in the Morse index
formula vanishes.
Applying \cite[Theorem~4.1]{APS} with
$\dim(S^d\times S^d)=2d$ and using
$\mu(x)=\mu_I(x)+n/2$, we obtain
\begin{equation}\label{eq:sphere-pinching-morse-grading}
 \begin{aligned}
 i_d(x)
 &=\left(\mu_I(x)+\frac12\right)
   +\frac{2d-(n+1)}2
 =\mu_I(x)+d-\frac n2,\\
 \mu(x)&=i_d(x)+n-d.
 \end{aligned}
\end{equation}

{To derive an analogue of
\eqref{eq:sphere-pinching-morse-grading} for $\mu_{-I}(x)$,
we consider normal variations of $c_x$.}
Let $\ell=\ell(x)$, and let $c:[0,\ell]\to S^{n+1}$
be the unit-speed parametrization of $c_x$.
At either endpoint, the velocity $v=\dot c$ satisfies
$(D\sigma)(v)=-v$.
By \eqref{eq:sphere-conjecture-symmetry}, the map
$\sigma:(S^{n+1},F)\to(S^{n+1},F^{\mathrm{rev}})$
is an isometry, where $F^{\mathrm{rev}}(q,w)=F(q,-w)$.
Since $g^{F^{\mathrm{rev}}}_{-v}=g^F_v$, the map
$D\sigma$ is orthogonal with respect to $g^F_v$ at
each endpoint.
Its $+1$ eigenspace is the tangent space of $S^d$,
whereas $v$ belongs to its $-1$ eigenspace.
In particular, $c$ is $g_{\dot c}$-orthogonal to $S^d$
at both endpoints.

The defining properties of the Chern connection also give
\[
 \nabla^{F^{\mathrm{rev}},-v}=\nabla^{F,v};
\]
see \cite[Theorem~2.4.1]{BCS} and
\cite[Section~2.3]{JavaloyesSoares} for these properties.
For vector fields $X,Y$ tangent to $S^d$, naturality under
the isometry $\sigma$, together with $(D\sigma)(v)=-v$,
therefore gives
\[
 (D\sigma)(\nabla^v_XY)=\nabla^v_XY,
 \qquad \nabla^v=\nabla^{F,v}.
\]
Thus $\nabla^v_XY$ is tangent to $S^d$, and the second
fundamental form terms in the index form vanish.

Let $H_t$ be the orthogonal complement of $\dot c(t)$
with respect to $g_{\dot c(t)}$, and set
\[
 U_0=T_{c(0)}S^d\subset H_0,\qquad
 U_\ell=T_{c(\ell)}S^d\subset H_\ell.
\]
For $t=0,\ell$, the $+1$ and $-1$ eigenspaces of
$(D\sigma)_{c(t)}$ in $H_t$ are $U_t$ and $U_t^\perp$,
respectively.
Here the orthogonal complements are taken in $H_t$.

By the second variation formula with endpoints on
submanifolds
\cite[Corollary~3.8 and Section~3.3]{JavaloyesSoares}
and the vanishing of the boundary terms, the index form
on normal variations is given by
\[
 Q_x(V,V)=\int_0^\ell
 \bigl(|D_tV|^2-\langle R_tV,V\rangle\bigr)\,dt.
\]
Here $D_t$ denotes Chern covariant differentiation along
$c$ with reference vector $\dot c(t)$, and $R_t$ is the
curvature endomorphism in the Jacobi equation
$D_t^2V+R_tV=0$.
The norm and inner product are taken with respect to
$g_{\dot c(t)}$.
We use the continuous extension of this quadratic form
to $H^1$ normal vector fields.

Let $i_+(x)$ be the Morse index of $Q_x$ on $H^1$ normal
vector fields satisfying
\begin{equation}\label{eq:sphere-plus-boundary}
 V(0)\in U_0,\qquad V(\ell)\in U_\ell,
\end{equation}
and let $i_-(x)$ be its Morse index on $H^1$ normal
vector fields satisfying
\begin{equation}\label{eq:sphere-minus-boundary}
 V(0)\in U_0^\perp,\qquad V(\ell)\in U_\ell^\perp.
\end{equation}
Every admissible vector field for the endpoint problem on
$S^d$ splits into a normal field satisfying
\eqref{eq:sphere-plus-boundary} and a tangential field
$f\dot c$ with $f(0)=f(\ell)=0$.
These two summands are orthogonal with respect to the
full index form, whose value on $f\dot c$ is
\[
 \int_0^\ell |f'(t)|^2\,dt\ge0.
\]
Hence $i_+(x)=i_d(x)$.

The energy functional is invariant under the loop involution
$\eta(t)\mapsto\sigma(\eta(2\ell-t))$.
At the closed geodesic obtained by reflection, the induced
involution on $H^1$ normal variations has $+1$ and $-1$
eigenspaces that are orthogonal with respect to the index form.
Restriction to $[0,\ell]$ identifies these eigenspaces with
the spaces satisfying \eqref{eq:sphere-plus-boundary} and
\eqref{eq:sphere-minus-boundary}, respectively.
Indeed, the inverse maps are given by the extensions
\[
 V(2\ell-t)=\pm(D\sigma)_{c(t)}V(t),
 \qquad 0\le t\le\ell,
\]
and the corresponding endpoint conditions ensure that
these extensions are continuous and periodic.
On each eigenspace, the closed-geodesic index form is
twice $Q_x$.
Tangential periodic variations have nonnegative index form.
The Morse index of the reflected closed geodesic on the
free loop space is therefore $i_+(x)+i_-(x)$.

We apply \cite[Remark~3.4 and Theorem~4.1]{APS} to the
contractible closed Reeb orbit $x^2$, with the capping
provided by Lemma~\ref{lem:brake-capped-iteration}.
The trivialization of the contact distribution, together
with the frame $(Z,X_H)$, extends over this capping.
A frame over the projected disk gives a trivialization
over the capping that preserves the vertical subbundle.
Their transition loop is therefore contractible, so the
two trivializations give the same Robbin--Salamon index
for the graph of the linearized flow.

In the symplectic frame $(Z,X_H)$, the linearized
Hamiltonian flow on the plane spanned by these vectors is
represented by the matrix
\[
 S(t)=\begin{pmatrix}1&0\\ t&1\end{pmatrix}.
\]
Using coordinates $(q,p)$ in this frame, its graph relative
to the diagonal $\Delta$, with the anti-diagonal as a
complementary Lagrangian, is represented by the quadratic
form $(q,p)\mapsto tq^2$.
The localization property in \cite[Theorem~2.3]{RS}
therefore gives
\[
 \mu_{\RS}(\operatorname{Gr}S,\Delta)=\frac12.
\]
The linearized Poincar\'e return map of $x^2$ has
no eigenvalue $1$.
On the plane spanned by $Z$ and $X_H$, the fixed subspace of
the linearized Hamiltonian return map is $\mathbb R X_H$.
Thus the periodic solutions of the full linearized
Hamiltonian equation are the constant multiples of $X_H$
along $x^2$, and the nullity is one.
The Morse index formula now gives
\[
 i_+(x)+i_-(x)
 =\left(\mu_{\mathrm{CZ}}(x^2)+\frac12\right)-\frac12
 =\mu_{\mathrm{CZ}}(x^2).
\]
Combining this identity with \eqref{eq:real-double-index}
and \eqref{eq:sphere-pinching-morse-grading}, we obtain
\begin{equation}\label{eq:sphere-two-morse-indices}
 i_+(x)=\mu_I(x)+d-\frac n2,
 \qquad
 i_-(x)=\mu_{-I}(x)+(n-d)-\frac n2.
\end{equation}

{To prove \eqref{eq:sphere-boundary-index-bound},
choose $\kappa>0$ such that}
\begin{equation}\label{eq:sphere-pinching-kappa}
 \left(\frac{\lambda}{\lambda+1}\right)^2
 <\kappa\le K_F\le1.
\end{equation}
Since $S^{n+1}$ is compact and simply connected, this
curvature bound allows us to apply
\cite[Theorem~1]{Rademacher}.
Every closed geodesic therefore has Finsler length at
least $\pi(1+\lambda^{-1})$.
Reflecting $c_x$ using $\sigma$ gives a closed geodesic
of length $2\ell(x)$, so
\begin{equation}\label{eq:sphere-chord-length}
 \ell(x)\sqrt\kappa
 \ge\frac\pi2(1+\lambda^{-1})\sqrt\kappa>\frac\pi2.
\end{equation}

To estimate $i_+(x)$ and $i_-(x)$, choose parallel normal
fields $E$ with initial values in $U_0$ and $U_0^\perp$,
respectively, and set
\[
 V(t)=\cos\!\left(\frac{\pi t}{2\ell}\right)E(t).
\]
Since $V(0)=E(0)$ and $V(\ell)=0$, these fields satisfy
\eqref{eq:sphere-plus-boundary} and
\eqref{eq:sphere-minus-boundary}, respectively.
The curvature bound $R_t\ge\kappa\Id$ on $H_t$ and
\eqref{eq:sphere-chord-length} give
\[
 Q_x(V,V)\le
 \left(\frac{\pi^2}{4\ell^2}-\kappa\right)
 \int_0^\ell|V|^2\,dt<0
\]
for every nonzero field $V$ of this form.
The fields constructed from initial values in $U_0$
and $U_0^\perp$ form vector spaces of dimensions $d$
and $n-d$, respectively.
Hence $i_+(x)\ge d$ and $i_-(x)\ge n-d$.
Together with \eqref{eq:sphere-two-morse-indices}, these
inequalities prove \eqref{eq:sphere-boundary-index-bound}.

For each integer $m\ge1$, consider the geodesic obtained
by projecting $x^m$ to $S^{n+1}$.
Parametrize it by arc length on $[0,m\ell(x)]$ and choose
a parallel orthonormal frame of its normal bundle.
Multiply each vector field in this frame by
\[
 \sin\!\left(\frac{j\pi t}{m\ell(x)}\right),
 \qquad j\in\Z_{>0},\qquad
 \frac{j\pi}{m\ell(x)}<\sqrt\kappa.
\]
These fields vanish at both endpoints, so they satisfy
the boundary conditions used to define $i_d(x^m)$.
Using $R_t\ge\kappa\Id$ and the orthogonality of the
sine functions, we see that the index form is negative
definite on their span.
Consequently,
\begin{equation}\label{eq:sphere-pinching-iterate-index}
 i_d(x^m)\ge
 n\left(\left\lceil\frac{m\ell(x)\sqrt\kappa}{\pi}\right\rceil-1\right)
 \ge\frac{nm\ell(x)\sqrt\kappa}{\pi}-n.
\end{equation}
For the corresponding estimates for closed geodesics,
see \cite[Lemma~1]{Rademacher}.

Dividing \eqref{eq:sphere-pinching-iterate-index} by $m$
and letting $m\to\infty$, we obtain from
\eqref{eq:sphere-pinching-morse-grading},
\eqref{eq:brake-mean-definition}, and
\eqref{eq:sphere-chord-length} that
\begin{equation}\label{eq:sphere-pinching-uniform-mean}
 \widehat\mu(x)
 \ge\frac{n\ell(x)\sqrt\kappa}{\pi}
 \ge\frac n2(1+\lambda^{-1})\sqrt\kappa.
\end{equation}

If $x(0)=x(T)$, then $c_x$ is itself a closed geodesic,
so \cite[Theorem~1]{Rademacher} gives
$\ell(x)\ge\pi(1+\lambda^{-1})$.
The first inequality in
\eqref{eq:sphere-pinching-uniform-mean} therefore yields
\[
 \widehat\mu(x)\ge n(1+\lambda^{-1})\sqrt\kappa>n,
\]
which proves \eqref{eq:sphere-pinching-chord-mean}.

If \eqref{eq:sphere-pinching} also holds, we may choose
$\kappa$ in \eqref{eq:sphere-pinching-kappa} so that
\[
 \frac{(2a+1)^2}{(a+1)^2}
 \left(\frac{\lambda}{\lambda+1}\right)^2<\kappa.
\]
The second inequality in
\eqref{eq:sphere-pinching-uniform-mean} then gives
\eqref{eq:sphere-pinching-mean-threshold}.
\end{proof}

\begin{proof}[Proof of Theorem~\ref{cor:sphere-pinching}]
Suppose first that $a\ge1$.
By \eqref{eq:sphere-pinching-maslov},
we have $N_L=n$ and $b(L)=2a+2$.
When $n=2$, the required filling hypothesis follows from
Lemma~\ref{lem:sphere-three-filling}.
Lemma~\ref{lem:sphere-pinching-index-estimates} verifies
the index assumption of Theorem~\ref{thm:main-brake-critical}.
{Theorem~\ref{thm:main-brake-critical} therefore gives at least
$2a+1$ distinct prime $\sigma$-brake geodesics.}

If \eqref{eq:sphere-pinching} also holds, then
\eqref{eq:sphere-pinching-mean-threshold}, together with
$b(L)=2a+2$, gives \eqref{eq:main-brake-critical-mean}.
{Theorem~\ref{thm:main-brake-critical} then gives at least
$2a+2$ distinct prime $\sigma$-brake geodesics.}

It remains to consider $a=0$.
{Since antipodal invariance is not assumed, we use the RFH
groups \eqref{eq:sphere-component-rfh} of $\Lambda_\pm$
directly on $Y_F$.}

{Each component $\Lambda_\pm$ admits a positive chord.
If one did not, reversal would exclude negative chords
as well, so its filtered Floer complexes would have only
generators arising from constant chords, in degrees
$0,\ldots,n$.
This contradicts \eqref{eq:sphere-component-rfh}, which gives
nonzero RFH groups in arbitrarily large degrees.}

Suppose now that there is only one geometrically distinct
$\rho$-brake Reeb orbit, and let $\gamma$ be its prime
representative.
Since each component admits a positive chord,
$\gamma$ must meet both $\Lambda_+$ and $\Lambda_-$.
By \eqref{eq:two-brake-markings}, it has exactly two fixed
points of $\rho$, one in each component.
The shortest positive chord $x:[0,T]\to Y_F$ of
$\Lambda_+$ therefore traverses $\gamma$ once.
Every positive chord of $\Lambda_+$ is an iterate
$x^m$, $m\ge1$.

{Since $x(0)=x(T)$,
Lemma~\ref{lem:sphere-pinching-index-estimates} gives
$\widehat\mu(x)>n$.}
{By \eqref{eq:sphere-component-rfh}, for every fixed integer $J$,}
\[
 \sum_{j=J}^R\dim_\F\RFH_j(\Lambda_+;\F)
 =\frac Rn+O(1)
 \qquad(R\in\Z,\ R\to\infty).
\]
The prime orbit $\gamma$ is strongly nondegenerate by
hypothesis, and $\widehat\mu(x)>0$.
Thus Lemma~\ref{lem:rfh-mean-density} applies to
$\Lambda_+$ and its single family of positive chords
$x^m$, $m\ge1$, and gives
\[
 \frac1n\le\frac1{\widehat\mu(x)}.
\]
This contradicts $\widehat\mu(x)>n$.
There are therefore at least two geometrically distinct
$\rho$-brake Reeb orbits, and hence at least two distinct
prime $\sigma$-brake geodesics.
\end{proof}

\begin{example}\label{ex:sphere-katok-sharp}
Let $n\ge2$, and let $\sigma:S^{n+1}\to S^{n+1}$ be the
restriction of a nontrivial orthogonal linear involution of
$\R^{n+2}$, with $\Fix(\sigma)\ne\varnothing$.
Write $\Fix(\sigma)\cong S^d$, where $0\le d\le n$, and set
$a=\min\{d,n-d\}$.
There exist Randers metrics on $S^{n+1}$ satisfying
the hypotheses of Theorem~\ref{cor:sphere-pinching}
for the given involution $\sigma$, as well as the
pinching condition \eqref{eq:sphere-pinching}.
These metrics have exactly $2a+2$ distinct prime
$\sigma$-brake geodesics.
\end{example}

\begin{proof}
We use the Katok--Ziller construction \cite{ZillerKatok}.
Set $r=a+1$ and $k=n-2a$.
We also write $\sigma$ for its orthogonal linear extension
to $\R^{n+2}$.
Let $E_+$ and $E_-$ denote its $+1$ and $-1$ eigenspaces,
respectively.
Their dimensions are $d+1$ and $n-d+1$, respectively.
Choose orthonormal sets $e_1,\ldots,e_r$ in $E_+$ and
$f_1,\ldots,f_r$ in $E_-$, and set
\[
 P_j=\Span\{e_j,f_j\},\qquad
 K=(P_1\oplus\cdots\oplus P_r)^\perp.
\]
Then $\dim K=k$, and $K$ is contained in one of the
eigenspaces of $\sigma$.
Define the skew-symmetric linear map $A_0$ by
\[
 A_0e_j=jf_j,\qquad A_0f_j=-je_j,\qquad A_0|_K=0.
\]
Then $\sigma A_0\sigma=-A_0$.

Choose an irrational number $\varepsilon$ with
$0<\varepsilon<1/((a+1)(2a+1))$.
The vector field $W(q)=\varepsilon A_0q$ is a Killing
field for $g_0$, and
\[
 w:=\max_{q\in S^{n+1}}|W(q)|_{g_0}
   =r\varepsilon<\frac1{2a+1}\le1.
\]
Zermelo navigation from $g_0$ and $W$ therefore
defines a smooth, strongly convex Randers metric $F$;
see \cite[Proposition~1.1]{BRSNavigation}.
For each $q\in S^{n+1}$, the $F$-unit sphere in
$T_qS^{n+1}$ is
\[
 \{W(q)+u:u\in T_qS^{n+1},\ g_0(u,u)=1\}.
\]
Since $\sigma$ and the antipodal map preserve $g_0$,
the identities
$W(\sigma(q))=-(D\sigma)_q\bigl(W(q)\bigr)$ and
$W(-q)=-W(q)$ imply
\[
 F\bigl(\sigma(q),-(D\sigma)_q(v)\bigr)=F(q,v),
 \qquad F(-q,-v)=F(q,v).
\]
The metric $F$ has constant flag curvature one by
\cite[Theorem~3.1]{BRSNavigation}.
Its reversibility is $\lambda=(1+w)/(1-w)$, so
\[
 \frac{(2a+1)^2}{(a+1)^2}
 \left(\frac{\lambda}{\lambda+1}\right)^2
 =\left(\frac{(2a+1)(1+w)}{2(a+1)}\right)^2<1.
\]

To determine the closed geodesics, identify
$S_{g_0}S^{n+1}$ with $S_FS^{n+1}$ by
$(q,u)\mapsto(q,u+W(q))$.
Under this identification, the unit-speed geodesic flow
of $F$ is given by
\begin{equation}\label{eq:katok-geodesic-flow}
 (q(t),u(t))=
 \left(e^{\varepsilon tA_0}(\cos t\,q_0+\sin t\,u_0),
       e^{\varepsilon tA_0}(-\sin t\,q_0+\cos t\,u_0)\right);
\end{equation}
see \cite[Theorem~2]{RoblesNavigation}.
In these coordinates, $\dot q(t)=u(t)+W(q(t))$.
Set $z_0=q_0+iu_0$.
Equation~\eqref{eq:katok-geodesic-flow} gives
\[
 q(t)+iu(t)=e^{-it}e^{\varepsilon tA_0}z_0.
\]
Thus the geodesic is periodic with period $T>0$ if and only if
\[
 e^{\varepsilon TA_0}z_0=e^{iT}z_0.
\]
The eigenvalues of the complexification of $A_0$ are $im$
for $m\in\{\pm1,\ldots,\pm r\}$, with zero also occurring
when $k>0$.
If the component of $z_0$ in the eigenspace corresponding
to $im$ is nonzero, periodicity requires
$(1-m\varepsilon)T\in2\pi\Z$.
Since $\varepsilon$ is irrational, the ratio
$(1-m\varepsilon)/(1-m'\varepsilon)$ is irrational whenever
$m\ne m'$.
Thus, for a closed geodesic, $z_0$ must lie in a single
eigenspace.

If $z_0$ lies in an eigenspace with eigenvalue $ij$ or $-ij$,
then $q_0$ and $u_0$ lie in $P_j$, so the geodesic traces
the great circle $P_j\cap S^{n+1}$.
For each $j=1,\ldots,r$, the initial data
$(q_0,u_0)=(e_j,f_j)$ and $(q_0,u_0)=(e_j,-f_j)$
give prime closed $F$-geodesics that traverse this circle
in opposite directions.
Their respective prime periods and Finsler lengths are
\begin{equation}\label{eq:katok-prime-periods}
 T_j^+=\frac{2\pi}{1+j\varepsilon},
 \qquad
 T_j^-=\frac{2\pi}{1-j\varepsilon}.
\end{equation}

{If $z_0$ lies in the zero eigenspace, then $q_0,u_0\in K$,
and \eqref{eq:katok-geodesic-flow} gives the great circles
contained in $K\cap S^{n+1}$.}
They exist if and only if $k\ge2$, since $q_0$ and $u_0$
must be orthonormal vectors in $K$.

Both orientations of each circle $P_j\cap S^{n+1}$
are $\sigma$-brake geodesics.
With $e_j$ as the initial point, $\sigma$ sends the angular
coordinate $\theta$ to $-\theta$.
No geodesic contained in $K\cap S^{n+1}$ is a
$\sigma$-brake geodesic.
Indeed, if $\sigma|_K=-\Id$, there is no fixed point of
$\sigma$ in $K\cap S^{n+1}$.
If $\sigma|_K=\Id$, the brake condition would force the
initial velocity to vanish.
Thus $F$ has exactly $2r=2a+2$ distinct prime
$\sigma$-brake geodesics.

It remains to verify that these prime $\sigma$-brake
geodesics are strongly nondegenerate.
Consider either orientation of the great circle
$P_i\cap S^{n+1}$, and let $T=T_i^\pm$ be its prime period.
In the coordinates $(q,u)$, the tangent space to
$S_{g_0}S^{n+1}$ at a point of this orbit splits into
the flow direction and $P_i^\perp\oplus P_i^\perp$.
The linearized flow preserves the second summand,
on which we compute the linearized Poincar\'e return map.

By \eqref{eq:katok-geodesic-flow}, this return map is
the composition of the linearized geodesic flow of $g_0$
at time $T$ with $e^{\varepsilon TA_0}$ acting on
both components.
{For each $j\ne i$, the restriction of the return map
to $P_j\oplus P_j$}
has eigenvalues
\[
 e^{\pm i(1+j\varepsilon)T},\qquad
 e^{\pm i(1-j\varepsilon)T}.
\]
If $k>0$, its restriction to $K\oplus K$ has eigenvalues
$e^{iT}$ and $e^{-iT}$, each with multiplicity $k$.
Equation~\eqref{eq:katok-prime-periods} and the
irrationality of $\varepsilon$ imply that
\[
 \frac{(1+j\varepsilon)T}{2\pi},\qquad
 \frac{(1-j\varepsilon)T}{2\pi}
 \qquad(j\ne i)
\]
are irrational, as is $T/(2\pi)$.
{Since no eigenvalue of the return map is a root of unity,
the prime $\sigma$-brake geodesics are strongly nondegenerate.}
\end{proof}

\begin{remark}
The strong nondegeneracy assumption in
Theorem~\ref{cor:sphere-pinching} applies only to the prime
$\sigma$-brake geodesics.
In Example~\ref{ex:sphere-katok-sharp}, when $n-2a\ge3$,
the great circles contained in $K\cap S^{n+1}$ give
infinitely many geometrically distinct prime closed geodesics.
These geodesics vary in smooth families and are therefore
degenerate.
None of them is a $\sigma$-brake geodesic.
\end{remark}

\section*{Acknowledgments}

Peng Liu is supported by the Postdoctoral Fellowship Program of CPSF
under Grant Number GZC20252019.
Dong Xie is partially supported by NSFC
(12631007, 12271268 and 12361141812) and the Natural Science
Foundation of Tianjin (No. 25JCZDJC01030).

ChatGPT (OpenAI) was used to assist with revising the language of the
manuscript, searching the literature and checking references, and
reviewing mathematical arguments.
The authors take full responsibility for all statements, proofs,
and references in this paper.


\begingroup
\small
\begin{thebibliography}{BSWX26}
\bibitem[APS08]{APS}
A.~Abbondandolo, A.~Portaluri, and M.~Schwarz,
\emph{The homology of path spaces and Floer homology with conormal
boundary conditions}, J. Fixed Point Theory Appl. \textbf{4} (2008),
no.~2, 263--293.
\doilink{https://doi.org/10.1007/s11784-008-0097-y}{doi:10.1007/s11784-008-0097-y}.

\bibitem[AbM26a]{AM26a}
M.~Abreu and L.~Macarini,
\emph{Periodic orbits of non-degenerate lacunary contact forms on
prequantization bundles}, Selecta Math. (N.S.) \textbf{32} (2026),
no.~3, Paper No.~59.
\doilink{https://doi.org/10.1007/s00029-026-01169-3}
{doi:10.1007/s00029-026-01169-3}.

\bibitem[AbM26b]{AM26b}
M.~Abreu and L.~Macarini,
\emph{Multiplicity of closed Reeb orbits on contact manifolds with
periodic equivariant symplectic homology}, preprint, 2026,
\href{https://arxiv.org/abs/2607.01219v2}{arXiv:2607.01219v2}.

\bibitem[AlF09]{AFReebChords}
P.~Albers and U.~Frauenfelder,
\emph{Floer homology for negative line bundles and Reeb chords in
prequantization spaces}, J. Mod. Dyn. \textbf{3} (2009),
no.~3, 407--456.
\doilink{https://doi.org/10.3934/jmd.2009.3.407}
{doi:10.3934/jmd.2009.3.407}.

\bibitem[BKK26]{BKK}
H.~Bae, J.~Kang, and S.~Kim,
\emph{Rabinowitz Floer homology for Legendrian submanifolds in
prequantization bundles}, preprint, 2026,
\href{https://arxiv.org/abs/2606.31674v1}{arXiv:2606.31674v1}.

\bibitem[BSWX26]{BSWXAG}
S.~Bai, E.~Shelukhin, Y.~Wang, and G.~Xu,
\emph{A proof of the Arnold--Givental conjecture}, preprint, 2026,
\href{https://arxiv.org/abs/2608.27242v1}{arXiv:2608.27242v1}.

\bibitem[BCS00]{BCS}
D.~Bao, S.-S.~Chern, and Z.~Shen,
\emph{An Introduction to Riemann-Finsler Geometry},
Graduate Texts in Mathematics, vol.~200,
Springer-Verlag, New York, 2000.
\doilink{https://doi.org/10.1007/978-1-4612-1268-3}
{doi:10.1007/978-1-4612-1268-3}.

\bibitem[BRS04]{BRSNavigation}
D.~Bao, C.~Robles, and Z.~Shen,
\emph{Zermelo navigation on Riemannian manifolds},
J. Differential Geom. \textbf{66} (2004), no.~3, 377--435.
\doilink{https://doi.org/10.4310/jdg/1098137838}{doi:10.4310/jdg/1098137838}.

\bibitem[BiC09]{BC}
P.~Biran and O.~Cornea,
\emph{Rigidity and uniruling for Lagrangian submanifolds},
Geom. Topol. \textbf{13} (2009), 2881--2989.
\doilink{https://doi.org/10.2140/gt.2009.13.2881}{doi:10.2140/gt.2009.13.2881}.

\bibitem[BrC26]{BCReebChords}
F.~Bro\'ci\'c and D.~Cant,
\emph{Lengths of Reeb chords and Viterbo restriction},
preprint (2026),
\href{https://arxiv.org/abs/2606.16270}{arXiv:2606.16270}.

\bibitem[BCS25]{BCSConormal}
F.~Bro\'ci\'c, D.~Cant, and E.~Shelukhin,
\emph{The chord conjecture for conormal bundles},
Math. Ann. \textbf{392} (2025), 3959--4021.
\doilink{https://doi.org/10.1007/s00208-025-03171-0}
{doi:10.1007/s00208-025-03171-0}.

\bibitem[CiO18]{CO}
K.~Cieliebak and A.~Oancea,
\emph{Symplectic homology and the Eilenberg--Steenrod axioms},
Algebr. Geom. Topol. \textbf{18} (2018), 1953--2130.
\doilink{https://doi.org/10.2140/agt.2018.18.1953}{doi:10.2140/agt.2018.18.1953}.

\bibitem[Cor21]{CoronaDisk}
D.~Corona,
\emph{A multiplicity result for orthogonal geodesic chords in Finsler disks},
Discrete Contin. Dyn. Syst. \textbf{41} (2021), no.~11, 5329--5357.
\doilink{https://doi.org/10.3934/dcds.2021079}{doi:10.3934/dcds.2021079}.

\bibitem[Dav10]{DavisPPS}
D.~M.~Davis,
\emph{Projective product spaces},
J. Topol. \textbf{3} (2010), no.~2, 265--279.
\doilink{https://doi.org/10.1112/jtopol/jtq006}
{doi:10.1112/jtopol/jtq006}.

\bibitem[GGP20]{GGPSeifert}
R.~Giamb\`o, F.~Giannoni, and P.~Piccione,
\emph{Multiple orthogonal geodesic chords and a proof of Seifert's
conjecture on brake orbits}, preprint, 2020, revised 2022,
\href{https://arxiv.org/abs/2002.09687v3}{arXiv:2002.09687v3}.

\bibitem[GGM18]{GGM}
V.~L.~Ginzburg, B.~Z.~G\"urel, and L.~Macarini,
\emph{Multiplicity of closed Reeb orbits on prequantization bundles},
Israel J. Math. \textbf{228} (2018), 407--453.
\doilink{https://doi.org/10.1007/s11856-018-1769-y}{doi:10.1007/s11856-018-1769-y}.

\bibitem[IMMO16]{IMMO}
H.~Iriyeh, H.~Ma, R.~Miyaoka, and Y.~Ohnita,
\emph{Hamiltonian non-displaceability of Gauss images of isoparametric
hypersurfaces}, Bull. Lond. Math. Soc. \textbf{48} (2016),
no.~5, 802--812.
\doilink{https://doi.org/10.1112/blms/bdw040}{doi:10.1112/blms/bdw040}.

\bibitem[JaS15]{JavaloyesSoares}
M.~A.~Javaloyes and B.~L.~Soares,
\emph{Geodesics and Jacobi fields of pseudo-Finsler manifolds},
Publ. Math. Debrecen \textbf{87} (2015), no.~1--2, 57--78.
\doilink{https://doi.org/10.5486/PMD.2015.7028}
{doi:10.5486/PMD.2015.7028}.

\bibitem[Kan18]{KangReversible}
J.~Kang,
\emph{On reversible maps and symmetric periodic points},
Ergodic Theory Dynam. Systems \textbf{38} (2018), no.~4, 1479--1498.
\doilink{https://doi.org/10.1017/etds.2016.71}{doi:10.1017/etds.2016.71}.

\bibitem[KKK22]{KKK}
J.~Kim, S.~Kim, and M.~Kwon,
\emph{Equivariant wrapped Floer homology and symmetric periodic Reeb
orbits}, Ergodic Theory Dynam. Systems \textbf{42} (2022),
no.~5, 1708--1763.
\doilink{https://doi.org/10.1017/etds.2020.144}{doi:10.1017/etds.2020.144}.

\bibitem[Kim22]{KimSphere}
S.~Kim,
\emph{Symmetric periodic orbits and invariant disk-like global
surfaces of section on the three-sphere},
Trans. Amer. Math. Soc. \textbf{375} (2022), no.~6, 4107--4151.
\doilink{https://doi.org/10.1090/tran/8516}{doi:10.1090/tran/8516}.

\bibitem[LLZ21]{LLZSurvey}
C.~Liu, Y.~Long, and D.~Zhang,
\emph{Index iteration theory for brake orbit type solutions
and applications},
Anal. Theory Appl. \textbf{37} (2021), no.~2, 129--156.
\doilink{https://doi.org/10.4208/ata.2021.pr80.05}
{doi:10.4208/ata.2021.pr80.05}.

\bibitem[LiZ14a]{LZIteration}
C.~Liu and D.~Zhang,
\emph{Iteration theory of $L$-index and multiplicity of brake orbits},
J. Differential Equations \textbf{257} (2014), no.~4, 1194--1245.
\doilink{https://doi.org/10.1016/j.jde.2014.05.006}
{doi:10.1016/j.jde.2014.05.006}.

\bibitem[LiZ14b]{LZSeifert}
C.~Liu and D.~Zhang,
\emph{Seifert conjecture in the even convex case},
Comm. Pure Appl. Math. \textbf{67} (2014), no.~10, 1563--1604.
\doilink{https://doi.org/10.1002/cpa.21525}
{doi:10.1002/cpa.21525}.

\bibitem[LiZ19]{LZDynamic}
Z.~Liu and D.~Zhang,
\emph{Brake orbits on compact symmetric dynamically convex
reversible hypersurfaces on $\R^{2n}$},
Discrete Contin. Dyn. Syst. \textbf{39} (2019), no.~7, 4187--4206.
\doilink{https://doi.org/10.3934/dcds.2019169}
{doi:10.3934/dcds.2019169}.

\bibitem[LZZ06]{LZZBrake}
Y.~Long, D.~Zhang, and C.~Zhu,
\emph{Multiple brake orbits in bounded convex symmetric domains},
Adv. Math. \textbf{203} (2006), no.~2, 568--635.
\doilink{https://doi.org/10.1016/j.aim.2005.05.005}{doi:10.1016/j.aim.2005.05.005}.

\bibitem[MoL26]{MorenoLimoge}
A.~Moreno and A.~Limoge,
\emph{A relative Poincar\'e--Birkhoff theorem},
Bull. Lond. Math. Soc. \textbf{58} (2026), no.~7, e70426.
\doilink{https://doi.org/10.1112/blms.70426}{doi:10.1112/blms.70426}.

\bibitem[Oh95]{OhRealForms}
Y.-G.~Oh,
\emph{Floer cohomology of Lagrangian intersections and
pseudo-holomorphic disks, III: Arnold--Givental conjecture},
in \emph{The Floer Memorial Volume},
H.~Hofer, C.~H.~Taubes, A.~Weinstein, and E.~Zehnder (eds.),
Progress in Mathematics, vol.~133,
Birkh\"auser, Basel, 1995, 555--573.
\doilink{https://doi.org/10.1007/978-3-0348-9217-9_23}
{doi:10.1007/978-3-0348-9217-9\_23}.

\bibitem[Rab87]{RabinowitzBrake}
P.~H.~Rabinowitz,
\emph{On the existence of periodic solutions for a class of symmetric
Hamiltonian systems}, Nonlinear Anal. \textbf{11} (1987),
no.~5, 599--611.
\doilink{https://doi.org/10.1016/0362-546X(87)90075-7}
{doi:10.1016/0362-546X(87)90075-7}.

\bibitem[Rad07]{Rademacher}
H.-B.~Rademacher,
\emph{Existence of closed geodesics on positively curved Finsler manifolds},
Ergodic Theory Dynam. Systems \textbf{27} (2007), no.~3, 957--969.
\doilink{https://doi.org/10.1017/S0143385706001064}
{doi:10.1017/\allowbreak S0143385706001064}.

\bibitem[RoS93]{RS}
J.~Robbin and D.~Salamon,
\emph{The Maslov index for paths}, Topology \textbf{32} (1993),
no.~4, 827--844.
\doilink{https://doi.org/10.1016/0040-9383(93)90052-W}{doi:10.1016/0040-9383(93)90052-W}.

\bibitem[Rob07]{RoblesNavigation}
C.~Robles,
\emph{Geodesics in Randers spaces of constant curvature},
Trans. Amer. Math. Soc. \textbf{359} (2007), no.~4, 1633--1651.
\doilink{https://doi.org/10.1090/S0002-9947-06-04051-7}
{doi:10.1090/\allowbreak S0002-9947-06-04051-7}.

\bibitem[She26]{ShelukhinChord}
E.~Shelukhin,
\emph{On orderability and the chord conjecture}, preprint, 2026,
\href{https://arxiv.org/abs/2608.10493v1}{arXiv:2608.10493v1}.

\bibitem[Szu89]{SzulkinBrake}
A.~Szulkin,
\emph{An index theory and existence of multiple brake orbits for
star-shaped Hamiltonian systems}, Math. Ann. \textbf{283} (1989),
no.~2, 241--255.
\doilink{https://doi.org/10.1007/BF01446433}{doi:10.1007/BF01446433}.

\bibitem[Zil83]{ZillerKatok}
W.~Ziller,
\emph{Geometry of the Katok examples},
Ergodic Theory Dynam. Systems \textbf{3} (1983), no.~1, 135--157.
\doilink{https://doi.org/10.1017/S0143385700001851}
{doi:10.1017/\allowbreak S0143385700001851}.

\end{thebibliography}
\endgroup
\end{document}